\documentclass[a4paper,10pt,notitlepage,reqno]{article}
\usepackage[english]{babel}
\usepackage[latin1]{inputenc}
\usepackage[T1]{fontenc}
\usepackage{amssymb,amsthm,amsmath} 
\usepackage{mathtools}
\usepackage{mathrsfs}
\usepackage{enumitem}
\usepackage{authblk}
\usepackage{color}
\usepackage{geometry}
\usepackage{hyperref}
\theoremstyle{plain} 
\newtheorem{theorem}{Theorem}[section]
\newtheorem{lemma}{Lemma}[section]
\newtheorem{proposition}{Proposition}[section]

\theoremstyle{definition}

\theoremstyle{remark}
\newtheorem{remark}{Remark}[section]

\DeclarePairedDelimiter{\abs}{\lvert}{\rvert} 
\DeclarePairedDelimiter{\norm}{\lVert}{\rVert}  

\DeclareMathOperator{\divergenza}{div}
\renewcommand{\div}{\divergenza}

\DeclareMathOperator{\realpart}{Re}
\renewcommand{\Re}{\realpart}

\DeclareMathOperator{\imaginarypart}{Im}
\renewcommand{\Im}{\imaginarypart}

\DeclareMathOperator{\sgn}{sgn}
\DeclareMathOperator{\supp}{supp}

\newcommand{\R}{\mathbb{R}}
\newcommand{\C}{\mathbb{C}}

\newcommand{\normeq}[1]{{\left\vert\kern-0.25ex\left\vert\kern-0.25ex\left\vert #1 
    \right\vert\kern-0.25ex\right\vert\kern-0.25ex\right\vert}}

\newenvironment{system}%
{\left\lbrace\begin{array}{r@{\hspace{1mm}}ll}}%
{\end{array}\right.}

\usepackage{color}
\usepackage[normalem]{ulem}
\definecolor{DarkGreen}{rgb}{0,0.5,0.1} 

\newcommand\soutD{\bgroup\markoverwith
{\textcolor{DarkGreen}{\rule[.5ex]{2pt}{1pt}}}\ULon}

\title{\textbf{Absence of eigenvalues of non-self-adjoint Robin Laplacians 
on the half-space}}

\author[1]{Lucrezia Cossetti} 
\author[2]{David Krej\v{c}i\v{r}\'ik}

\affil[1,2]{Department of Mathematics, Faculty of Nuclear Sciences and Physical Engineering, Czech Technical University in Prague, Trojanova 13, 12000 Prague 2, Czechia; cosseluc@fjfi.cvut.cz, david.krejcirik@fjfi.cvut.cz}

\affil[1]{Current affiliation: Department of Mathematics, Institute for Analysis, Karlsruhe Institute of Technology, Englerstra{\ss}e 2, 76131 Karlsruhe, Germany; lucrezia.cossetti@kit.edu}

\begin{document}

\date{\small 22 November 2019}

\maketitle

\begin{abstract}
\noindent
By developing the method of multipliers, we establish sufficient conditions 
which guarantee the total absence of eigenvalues of the Laplacian in the half-space,
subject to variable complex Robin boundary conditions. 
As a further application of this technique, 
uniform resolvent estimates are derived under the same assumptions on the potential.     
Some of the results are new even in the self-adjoint setting,
where we obtain quantum-mechanically natural conditions.
\end{abstract}



\section{Introduction}
%
Proving the absence of eigenvalues 
of self-adjoint Schr\"odinger operators is intimately related 
to scattering theory in quantum mechanics and constitutes 
a by now classical research field of mathematical physics.
Recent years have brought new motivations for considering
non-self-adjoint operators in modern physics (including quantum mechanics)
and a lot of attention has been focused on  
Schr\"odinger operator with possibly complex potentials
which attracted little attention earlier.

An important exception from this state of the art is 
the 1966 Kato's paper \cite{Kato_1966}, where an abstract
method based on the stationary scattering theory
was developed for proving, among other things,
the (total) absence of eigenvalues of Schr\"odinger operators 
in three and higher dimensions with suitably \emph{small} complex potentials. 
For \emph{discrete eigenvalues} a related 
(and stronger in an $L^p$ scale) result was obtained
by Frank in 2011 \cite{Frank_2011}
(see also \cite{Frank-Simon_2017} for the inclusion of embedded eigenvalues)
with help of the Birman-Schwinger principle
and uniform Sobolev inequalities obtained in
\cite{Kenig-Ruiz-Sogge_1987}.
Yet another approach based on the method of multipliers
\emph{\`a la} Morawetz~\cite{Morawetz_1968} was developed in~\cite{FKV},
where Fanelli, Vega and the second author went
beyond the smallness restriction and included possibly large
repulsive potentials, too.  
Moreover, in~\cite{FKV2} the two-dimensional 
electromagnetic Schr\"odinger operators has been covered by the same authors. 
We also mention~\cite{Cossetti_2017}
where an adaptation of this method to an elasticity setting
was performed by the first author. 

The origin of the present work was to continue with the research 
of \cite{FKV,FKV2,Cossetti_2017} by investigating the robustness
of the method of multipliers in the setting of Schr\"odinger operators
acting in \emph{subdomains} of the Euclidean space.
It turns out that the generalisation is not obvious
due to the presence of boundaries,
at least in the case of complex potentials.
Nonetheless, in this paper we show how to adapt the method
in the special setting of the half-space
and obtain physically relevant conditions of the same nature as above.
In order to highlight the role of boundaries, 
we decided to consider the field-free Schr\"odinger operator,
but subject to general complex Robin boundary conditions.
In this way, we are concerned with a sort of complex potential 
supported by a hyperplane of the space.

To state our main results,  
let us consider the upper half-space
\begin{equation}\label{upper}
  \Omega := \R^{n-1} \times (0,\infty)
\end{equation} 
with any $n \geq 1$,
so that the boundary~$\partial\Omega$ can be identified 
with the lower-dimensional Euclidean space~$\R^{n-1}$.
Given an arbitrary complex-valued function
$\alpha: \partial\Omega \to \C$ 
such that $\alpha \in L^\infty(\partial \Omega)$,
let $-\Delta_\alpha^\Omega$ be the m-sectorial operator 
in the Hilbert space $L^2(\Omega)$ associated with
the closed form
\begin{equation}\label{form}
  h_\alpha[u] := \int_\Omega |\nabla u|^2 
  + \int_{\partial\Omega} \alpha \, |u|^2
  \,, \qquad
  u \in \mathcal{D}(h_\alpha) := H^1(\Omega)
  \,.
\end{equation} 
It can be shown (see Appendix~\ref{Appendix:rigorous_definition})
that $-\Delta_\alpha^\Omega$ acts as the Laplacian in~$\Omega$
and satisfies the boundary condition 
$$
  -u_{x_n} + \alpha \;\! u = 0
  \qquad \mbox{on} \qquad 
  \partial\Omega
$$
in the sense of traces,
where $u_{x_{n}}$ denotes the first partial derivative 
with respect to the $n$th variable in~$\Omega$.
We consistently write $x=(x',x_n)\in \R^{n-1}\times \R$
with $x' := (x_1,\dots,x_{n-1})$
for a generic point~$x$ of~$\R^n$.

In the self-adjoint setting, we prove the following robust result.
\begin{theorem}\label{thm:self-adjoint}
	Let $n\geq 1$ and and assume that 
$\alpha\in W^{1,\infty}(\partial \Omega;\R)$ is such that 
	\begin{equation}\label{hp:self-adjoint}
		\alpha \geq 0,
	\end{equation} 
and
\begin{equation}\label{hp:repulsivity} 
	x\cdot \nabla \alpha \leq 0.
\end{equation}
Then $\sigma_\mathrm{p}(-\Delta_\alpha^{\Omega})=\varnothing$.
\end{theorem}   

Here $\sigma_\mathrm{p}(H)$ denotes the point spectrum of a closed operator~$H$,
that is, the set of eigenvalues of~$H$.
We underline that the result therefore
guarantees the \emph{total} absence of eigenvalues. 
More specifically, hypothesis~\eqref{hp:self-adjoint} is an elementary way 
how to exclude negative (discrete) eigenvalues 
(in fact any negative spectrum of $-\Delta_\alpha^\Omega$), 
while the non-trivial absence of non-negative (embedded) eigenvalues 
is guaranteed by the repulsivity condition~\eqref{hp:repulsivity}. 
The terminology for the classification of eigenvalues in the brackets
is due to the fact that,
under the hypotheses~\eqref{hp:self-adjoint} and~\eqref{hp:repulsivity},
one has 
\begin{equation}\label{spec}
  \sigma(-\Delta_\alpha^\Omega)=[0,+\infty) .
\end{equation}
We also remark that the spectrum is purely continuous
since, in addition to the absence of eigenvalues,
the residual spectrum is always empty in the self-adjoint case. 
Hence, the spectrum shares the same properties 
with the spectrum of the unperturbed situation $\alpha=0$
corresponding to the Neumann Laplacian in the half-space.

\begin{remark}
For $n=1$, the Robin problem is reduced to the half-line $\Omega=(0,\infty).$ In this situation the absence of eigenvalues is obtained just by requiring~\eqref{hp:self-adjoint} and this condition is known to be also necessary.  
\end{remark}  

If $n \geq 2$, even the self-adjoint setting of Theorem~\ref{thm:self-adjoint} 
seems to be new.
In~\cite{Beltita_1998} Beltita applied the Mourre theory 
(which is closely related to the technique 
behind the proof of our Theorem~\ref{thm:self-adjoint},
see Remark~\ref{Rem.Mourre})
to self-adjoint Schr\"odinger operators on the half-space, 
subject to Robin boundary conditions,
and obtained results about the nature of the essential spectrum
(location of possible accumulations of eigenvalues, 
absence of singularly continuous spectrum),
it is true, but the approach did not allow her to guarantee
the (total) absence of eigenvalues. 
Let us also mention the work of Frank~\cite{Frank_2006}, 
where he considered the special case of \emph{periodic}~$\alpha$
and established sufficient conditions which guarantee 
that the spectrum is purely absolutely continuous. 
In view of~\eqref{hp:repulsivity},
our setting is rather complementary to the periodic one.
 
The half-space~\eqref{upper} can be regarded as a degenerate situation
of conical domains intensively studied in recent years.
In~\cite{Ben_Dhia-Fliss-Hazard-Tonnoir_2016}
the authors established an interesting
Rellich-type result which guarantees the absence of  
non-negative eigenvalues of the Laplacian in non-convex domains of the form
\begin{equation}\label{Sophie}
\big\{(x,y)\in \R^{n-1}\times \R \, 
:\  y>- \abs{x}\tan \theta \big\}, \qquad \theta \in (0, \pi/2),
\end{equation}
subject to \emph{no specific boundary conditions}.
We underline that being $\theta>0$ this result does not cover the half-space, in fact on the contrary, without prescribing any specific behavior of the solutions on the boundary, it is easy to construct square-integrable solutions to the eigenvalues problem in a half-space. This fact, compared with our opposite result in Theorem~\ref{thm:self-adjoint}, stresses how spectral properties of such operators are strongly sensible to boundary conditions.
Interesting results have been also obtained for complements of~\eqref{Sophie},  
\emph{i.e.}\ on infinite sectors of~$\R^n.$ 
Among others works, let us quote~\cite{Khalile-Pankrashkin_2018} in which the discrete spectrum of self-adjoint Robin Laplacians is studied in the case of infinite \emph{planar} sectors. More precisely, 
assuming that~$\alpha$ is \emph{real, negative and constant}, 
it is proved that the discrete spectrum is non-empty if and only if the sector is strictly smaller than the half-plane.
 
In order to state our results in the non-self-adjoint setting, 
we need to recall two functional inequalities 
which enter the assumptions below.
First, let $n\geq 3$ and let $f$ and $g$ 
be complex-valued functions defined on~$\R^{n-1}$.
Then the fractional Leibnitz rule
\begin{equation}\label{eq:chain_rule}
\norm{(-\Delta)^{1/4}(f\;\! g)}_{L^2(\R^{n-1})}
\leq 
C \, \big[\norm{f}_{L^{\infty}(\R^{n-1})}\norm{(-\Delta)^{1/4} g}_{L^{2}(\R^{n-1})} 
+ \norm{(-\Delta)^{1/4} f}_{L^{2(n-1)}(\R^{n-1})}\norm{g}_{L^{2^\ast}(\R^{n-1})} \big]
\end{equation}
holds true with $2^\ast:=2(n-1)/(n-2)$
and a positive constant~$C$ depending on the dimension~$n$.
Here the action of the fractional Laplacian is understood
in a standard way via the Fourier transform (see, \emph{e.g.}, \cite[Sec.~7]{LL}) and it is also related to the homogeneous Sobolev space $\dot{H}^{1/2}(\R^{n-1})$ in the following way:
\begin{equation}\label{eq:homog-norm}
	\norm{(-\Delta)^{1/4} f}_{L^2(\R^{n-1})}=\norm{f}_{\dot{H}^{1/2}(\R^{n-1})}^2=\int_{\R^{n-1}} \abs{\xi} \abs{\hat{f}(\xi)}^2\, d\xi.
\end{equation}
Indeed, \eqref{eq:chain_rule}~is a particular case of the homogeneous 
Kato-Ponce inequality~\cite{Kenig-Ponce-Vega_1993}.
We denote by~$C^*$ the smallest constant~$C$ such that~\eqref{eq:chain_rule}   
holds true (unfortunately, the optimal constant does not seem 
to be known explicitly).
Second, the fractional Sobolev embeddings 
$\dot{H}^{1/2}(\R^{n-1})\hookrightarrow L^{2^\ast}(\R^{n-1})$ 
is quantified by the inequality
\begin{equation}\label{eq:Sobolev_inequality}
  \norm{f}_{L^{2^\ast}(\R^{n-1})}
  \leq S \, \norm{(-\Delta)^{1/4} f}_{L^2(\R^{n-1})}
  ,
\end{equation}
where~$S$ is a positive constant.
We denote by~$S^*$ the smallest constant~$S$ 
such that~\eqref{eq:Sobolev_inequality} holds true  
(in this case, the optimal constants is known explicitly,
see \cite[Thm.~8.4]{LL}).

\begin{theorem}\label{thm:non_self-adjoint}
Let $n\geq 3$ and assume that $\alpha\in L^{\infty}(\partial \Omega;\C)$ 
is such that  
	\begin{equation}\label{hp:non_self-adjoint} 
		\Re \alpha \geq \abs{\Im\alpha}.
	\end{equation}
If there exist non-negative numbers $b_1,b_2$ satisfying
	\begin{equation}\label{eq:smallness}
		2C^\ast \left[ b_1+S^\ast b_2 \right]<1
	\end{equation}
such that
\begin{equation}\label{eq:alpha_condition}
	\norm{x\, \alpha}_{L^\infty(\partial \Omega)}\leq b_1
\end{equation}
and 
\begin{equation}\label{eq:fractional_derivative}
	\norm{(-\Delta)^{1/4}(x\, \alpha)}_{L^{2(n-1)}(\partial \Omega)}\leq b_2,
\end{equation}	
	then $\sigma_\mathrm{p}(-\Delta_\alpha^\Omega)=\varnothing.$
\end{theorem}

Condition~\eqref{hp:non_self-adjoint} is a non-self-adjoint 
analogue of~\eqref{hp:self-adjoint} ensuring in a simple way
the absence of (discrete) eigenvalues in the exterior of the closed cone
$\mathcal{C} := \{\lambda\in\C \ |\,  \Re\lambda \geq |\Im\lambda|\}$.
On the other hand, as a consequence of a refined application
of the method of multipliers, 
\eqref{eq:alpha_condition} and~\eqref{eq:fractional_derivative}
guarantee that there are no eigenvalues (both discrete or embedded) in~$\mathcal{C}$. 
Note that~\eqref{eq:alpha_condition} implies that
$
  \sigma_\mathrm{ess}(-\Delta_\alpha^\Omega)=[0,+\infty)
$,
where one may take any of the customary definitions 
of the essential spectrum for non-self-adjoint operators
(\emph{cf.}~\cite[Sec.~IX]{Edmunds-Evans}).
At the same time, 
the residual spectrum is always empty as a consequence of the symmetry relation
$(-\Delta_\alpha^\Omega)^* = -\Delta_{-\alpha}^\Omega$.
Consequently, the stability result~\eqref{spec} again holds
and the spectrum is purely continuous.
 
In higher dimensions, we have a more explicit result.  
\begin{theorem}\label{thm:non_self-adjoint_h}
Let $n\geq 4$ and assume that $\alpha\in W^{1,\infty}(\partial \Omega;\C)$
is such that~\eqref{hp:non_self-adjoint}  
	and 
	\begin{equation}\label{hp:repulsivity_h}
		x\cdot \nabla\, \Re \alpha\leq 0.
	\end{equation}
	If there exist non-negative numbers $b_1,b_2$ satisfying
	\begin{equation}\label{eq:smallness_h}
		2\left[ b_1(b_1+ b_2) \right]<1,
	\end{equation}
	such that
\begin{equation}\label{eq:alpha_condition_h}
	\norm{x\, \Im \alpha}_{L^\infty(\partial \Omega)}\leq b_1 
\end{equation}
and 
\begin{equation}\label{eq:Hardy_h}
  \forall \psi\in H^1(\partial\Omega) , \qquad
	 \int_{\partial \Omega} \abs{\div(x' \Im \alpha)}^2 \abs{\psi}^2
  \leq b_2^2 \int_{\partial \Omega} \abs{\nabla \psi}^2 ,  
\end{equation}	
	then $\sigma_\mathrm{p}(-\Delta_\alpha^\Omega)=\varnothing.$
\end{theorem}
\begin{remark}
Formally, the theorem holds also for dimensions $n = 2,3$,
but then the statement reduces to the self-adjoint situation
of Theorem~\ref{thm:self-adjoint}
(under a stronger hypothesis about the regularity of~$\alpha$).
Indeed, the criticality of the Laplacian in dimensions $1,2$
(note that working on the boundary sets us in an $(n-1)$-dimensional context)
implies that condition~\eqref{eq:Hardy_h} can be satisfied only
in the trivial case $\div(x' \Im\alpha)=0$.
But the only admissible solution~$\alpha$ to this equation
is that with $\Im\alpha=0$, in which case 
\eqref{eq:alpha_condition_h} and~\eqref{eq:Hardy_h} 
are trivially satisfied and~\eqref{hp:non_self-adjoint} 
reduces to~\eqref{hp:self-adjoint}.    
On the other hand, hypothesis~\eqref{eq:Hardy_h} is non-void
as long as $n \geq 4$ due to the classical Sobolev embedding 
$\dot{H}^1(\R^{n-1})\hookrightarrow L^{2^*}(\R^{n-1})$
with $2^* := 2(n-1)/(n-3)$.
The latter is quantified by the inequality
\begin{equation}\label{eq:standard-Sobolev}
  \norm{f}_{L^{2^*}(\R^{n-1})}
  \leq \mathcal{S} \, \norm{\nabla f}_{L^2(\R^{n-1})}
\end{equation} 
valid for every $f \in \dot{H}^1(\R^{n-1})$,
where $\mathcal{S}$ is a positive constant
(the optimal constant~$\mathcal{S}^*$ is known explicitly, 
see~\cite[Thm. 8.3]{LL}).
Then a sufficient condition which guarantee~\eqref{eq:Hardy_h} 
is represented by
\begin{equation}\label{eq:L^(n-1)}
  \norm{\div(x' \Im\alpha)}_{L^{n-1}(\R^{n-1})}\leq b ,
\end{equation}
where~$b$ is a positive constant.
Indeed, by H\"older's inequality and~\eqref{eq:standard-Sobolev},
one gets
\begin{equation*}
\begin{split}
\int_{\partial \Omega} \abs{\div(x' \Im \alpha)}^2 \abs{\psi}^2		
&\leq \left( \int_{\partial \Omega} \abs{\div(x' \Im\alpha)}^{(n-1)} \right)^{2/(n-1)} \left(\int_{\partial \Omega} \abs{\psi}^{2(n-1)/(n-3)} \right)^{(n-3)/(n-1)}
\\
&\leq (\mathcal{S}^*b)^2 \int_{\partial \Omega} \abs{\nabla \psi}^2.  
\end{split}
\end{equation*} 
Hence~\eqref{eq:Hardy_h} is valid by choosing $b_2:=\mathcal{S}^* b$.
\end{remark}

In the non-self-adjoint context of 
Theorems~\ref{thm:non_self-adjoint} and~\ref{thm:non_self-adjoint_h},
it is necessary to mention the recent work of Frank~\cite{Frank_2017},
where he also establishes sufficient conditions which guarantee
the absence of eigenvalues of the Robin Laplacian in the half-space
(in fact, more generally, for Schr\"odinger operators on the whole space 
with $\delta$-type potentials supported on a hyperplane,
\emph{cf.}~\cite[Rem.~5]{Frank_2017}).
His approach is based on the Birman-Schwinger principle
as previously applied to Schr\"odinger operators in the whole space
in \cite{Frank_2011,Frank-Simon_2017} 
(see also~\cite[Sec.~2]{FKV}),
which yields the sufficient condition  
\begin{equation}\label{Frank}
  \int_{\partial\Omega} |\alpha|^{n-1} < D_{0,n}^{-1}
  \,, 
\end{equation}
where $D_{0,n}$ is a (non-explicit) positive constant as long as $n \geq 3$. 
This inequality is clearly a smallness-type condition 
of similar nature as our Theorem~\ref{thm:non_self-adjoint}
(obtained by a completely different approach),
but either of the two alternative hypotheses~\eqref{Frank} 
and \eqref{hp:non_self-adjoint}--\eqref{eq:fractional_derivative}
does not seem to imply the other in general.
Moreover, in Theorem~\ref{thm:non_self-adjoint_h} 
we go beyond the smallness of~$|\alpha|$.
Nevertheless, it is remarkable that the $L^{n-1}$-scale of~\eqref{Frank}
appears in our sufficient condition~\eqref{eq:L^(n-1)}, too.

\medskip
As a further application of the technique of multipliers developed to prove 
Theorems~\ref{thm:self-adjoint}, \ref{thm:non_self-adjoint} 
and~\ref{thm:non_self-adjoint_h}, we are also able 
to perform uniform resolvent estimates for $-\Delta_\alpha^\Omega$, 
which represent an analogue of the results obtained  
for electromagnetic Schr\"odinger operators in the whole space
in~\cite{Barcelo-Vega-Zubeldia_2013}.
\begin{theorem}\label{thm:main_resolvent}
Let $n\geq 3$ and assume that~$\alpha$ satisfies 
the hypotheses of Theorem~\ref{thm:non_self-adjoint}.
Then there exists a positive constant~$c$ such that,
for all $\lambda \in \C$,
\begin{equation*}
  \big\|
  r^{-1}(-\Delta_\alpha^\Omega-\lambda)^{-1}r^{-1}
  \big\|_{L^2(\Omega) \to L^2(\Omega)} \leq c
  \,,
\end{equation*}
where $r(x) := |x|.$
\end{theorem}

Actually, Theorem~\ref{thm:main_resolvent} is a direct consequence
of the following stronger result, 
which shows that \emph{a priori} estimates for solutions 
to the resolvent equation hold.

\begin{theorem}\label{thm:resolvent}
Let $n\geq 3$ and assume that~$\alpha$ satisfies 
the hypotheses of Theorem~\ref{thm:non_self-adjoint}.
There exists a positive constant~$c$ such that,
given any $\lambda \in \C$ 
and $f \in L^2(\Omega,|x|\,dx)$,
any solution $u \in \mathcal{D}(-\Delta_\alpha^\Omega)$
of the equation $(-\Delta_\alpha^\Omega-\lambda)u=f$ satisfies
				\begin{equation}\label{eq:res_1}
					\norm{\nabla u^-}_{L^2(\Omega)}\leq c \, \norm{r f}_{L^2(\Omega)},
					\qquad \text{for}\quad \abs{\Im\lambda}\leq \Re \lambda,
				\end{equation}
				and
				\begin{equation}\label{eq:res_2}
					\; \,\norm{\nabla u}_{L^2(\Omega)}\leq c \, \norm{r f}_{L^2(\Omega)},
					\qquad \text{for}\quad \abs{\Im\lambda}> \Re \lambda,
				\end{equation}
Here the auxiliary function $u^-$ is defined as follows 
	\begin{equation}\label{eq:u^-}
		u^-(x):= e^{-i (\Re \lambda)^{1/2} \sgn(\Im \lambda) \abs{x}} \, u(x).
	\end{equation}
\end{theorem} 
As a conclusion we stress that even though the proofs of our results are strongly based on the geometry we decided to focus on, namely the half-space domain, the underlying strategy we use is rather general and, at some extends, open to generalisations to different geometries. To that end, along the paper we tried to emphasize where the role of the geometry enters into play and, when applicable, we provided ideas for a more general investigation.

The paper is organised as follows. 
In Section~\ref{Sec.pre} we collect crucial identities  
on which the technique of multipliers is based.
In order to lighten the presentation, 
we only sketch the main ideas and 
the complete proof is postponed
to Appendix~\ref{Appendix:5_identities}. 
The technique of multipliers is then fully developed in Section~\ref{section:proofs},
where we provide proofs of the main theorems presented in the introduction. 
In Appendix~\ref{Appendix:rigorous_definition} 
we provide details on the definition of the Robin Laplacian $-\Delta_\alpha^\Omega$
and characterise its operator domain.

\section{Preliminaries}\label{Sec.pre}
In this preliminary section we develop the method of multipliers 
to provide some integral identities for solutions~$u$ to 
the boundary-value problem
\begin{equation}\label{eq:resolvent_gen}
	\begin{system}
		-\Delta u&= \lambda u +f \quad& \text{in}\quad \Omega,\\
		-u_{x_{n}} + \alpha u&=g &\text{on}\quad \partial \Omega,
	\end{system}
\end{equation}
where $f\colon \Omega \to \C$, $g\colon \partial \Omega \to \C$
and $\alpha\colon \partial \Omega \to \C$ are measurable functions
and $\lambda \in \C$.
We shall restrict to solutions belonging to the space
\begin{equation*}
  \widetilde{\mathcal{D}}_0
  :=\big\{
  u\in H^{3/2}(\Omega)\,\big|\, \Delta u \in L^2(\Omega) \ \& \
  \supp u \mbox{ is compact in } \R^n 
  \big\}
\end{equation*}
and consider the weak formulation of~\eqref{eq:resolvent_gen}.
More specifically, we say that
a compactly supported~$u$ solves~\eqref{eq:resolvent_gen} 
if $u \in \widetilde{\mathcal{D}}_0$ and the identity
\begin{equation}\label{eq:weak_resolvent_gen}
	\int_{\Omega} \nabla u \cdot \nabla \bar{v}\, dx + \int_{\partial \Omega} \alpha u \bar{v}\, d\sigma=\lambda \int_{\Omega} u \bar{v}\, dx + \int_{\Omega} f \bar{v}\, dx + \int_{\partial \Omega} g \bar{v}\, d\sigma
\end{equation}
holds true for every $v\in H^1(\R^n)$.
Here $d\sigma$ is the surface measure on $\partial \Omega$
and the boundary values are understood in the sense of traces. 
 
The method of multipliers is based on producing several integral identities
by choosing various test functions~$v$ in~\eqref{eq:weak_resolvent_gen}
and later combining them in a refined way.
The compact support requirement is just a technical condition 
for the justification of the algebraic manipulations 
in which the test function~$v$ involve unbounded multiples of~$u$
and its derivatives (so that it is not \emph{a priori} clear 
that the particular choice~$v$ belongs to $H^1(\R^n)$). 
However, in Section~\ref{section:proofs},
we shall see that \emph{any} weak solution to the eigenvalue problem
\begin{equation}\label{eq:Schr_half_plane}
	\begin{system}
		-\Delta u&=\lambda u \quad &\text{in}\quad \, \,\, \Omega,\\
		-u_{x_{n}} + \alpha  u&=0  &\text{on}\quad \partial \Omega,
	\end{system}
\end{equation}
can be approximated by a sequence of compactly supported functions solving an approximating boundary-value problem of the type~\eqref{eq:resolvent_gen}, 
with corrections~$f$ and~$g$ small in a suitable topology. 
This will enable us to exploit the results of this section to get information also for not necessarily compactly supported solutions to~\eqref{eq:Schr_half_plane}.

We stress that functions $u \in \widetilde{\mathcal{D}}_0$ 
are required to have a compact support in~$\R^n$ 
and not necessarily in~$\Omega$. 
This involves the convention that we use the same symbol~$u$ 
for $u\in H^{s}(\Omega)$ with $s\geq 0$ 
and its extension $Eu\in H^{s}(\R^n)$,
where $E: H^{s}(\Omega) \to H^{s}(\R^n)$ is an extension operator
(see, \emph{e.g.}, \cite[Sec.~5.17]{Adams2}).
Though the existence of such an extension operator $E$ for a general domain~$\Omega$ is not trivial and deeply depends on the geometry of~$\Omega$, 
it is a classical result that in our particular situation (half-space)
the extension operator can be built making use of a suitable reflection  argument.
Similarly, given a function~$u$ from the space 
\begin{equation}\label{space}
  \widetilde{\mathcal{D}}
  := \big\{
  u\in H^{3/2}(\Omega)\,\big|\, \Delta u \in L^2(\Omega)  
  \big\}
\end{equation}
and its derivative~$u_{x_n}$ with respect to the last variable,
we respectively use the same symbols~$u$ and~$u_{x_n}$
for the Dirichlet trace $\gamma_D u$ 
and the Neumann trace $\gamma_D u_{x_n}$,
where $\gamma_D:\widetilde{\mathcal{D}}\to L^2(\partial\Omega)$
is the Dirichlet trace map
(see Appendix~\ref{Appendix:rigorous_definition} for more details).

The crucial aforementioned integral identities 
are collected in the following lemma.

\begin{lemma}\label{lemma:5identities}
Let $n\geq 2$
and assume that $f\in L^2(\Omega)$ and $g\in L^2(\partial\Omega).$ 
Then any weak solution $u\in \widetilde{\mathcal{D}}_0$ 
of~\eqref{eq:resolvent_gen} satisfies the following five identities:
\begin{equation}\label{eq:first_Omega_const}
		\int_{\Omega} \abs{\nabla u}^2 \,dx + \int_{\partial \Omega} \Re \alpha \abs{u}^2\, d\sigma\\
	  = \Re \lambda \int_{\Omega} \abs{u}^2\, dx + \Re \int_{\Omega} {f} \bar{u}\, dx + \Re \int_{\partial \Omega} {g} \bar{u}\, d\sigma,
	\end{equation}
	\begin{multline}\label{eq:first_Omega_x}
		-\frac{n-1}{2} \int_{\Omega} \frac{~\abs{u}^2}{\abs{x}} \,dx + \int_{\Omega} \abs{x} \abs{\nabla u}^2 \,dx + \int_{\partial \Omega} \Re \alpha \abs{x} \abs{u}^2\, d\sigma\\
	  = \Re \lambda \int_{\Omega} \abs{x} \abs{u}^2\, dx + \Re \int_{\Omega} \abs{x} {f}\bar{u}\, dx + \Re \int_{\partial \Omega} \abs{x} {g} \bar{u}\, d\sigma,
	\end{multline}
	\begin{equation}\label{eq:first_Omega_im_const}
			\int_{\partial \Omega} \Im \alpha \abs{u}^2\, d\sigma= \Im \lambda \int_{\Omega} \abs{u}^2\, dx
			+ \Im \int_{\Omega} f \bar{u}\, dx + \Im \int_{\partial \Omega} {g} \bar{u}\, d\sigma,
	\end{equation}

	\begin{equation}\label{eq:first_Omega_im_x}
			\Im \int_{\Omega} \frac{x}{\abs{x}} \cdot \bar{u}\nabla u\, dx + \int_{\partial \Omega} \Im \alpha \abs{x} \abs{u}^2\, d\sigma= \Im \lambda \int_{\Omega} \abs{x} \abs{u}^2\, dx
			+ \Im \int_{\Omega} \abs{x} f \bar{u}\, dx + \Im \int_{\partial \Omega} \abs{x} {g} \bar{u}\, d\sigma,
	\end{equation}
	\begin{multline}\label{eq:second_Omega}
		2 \int_{\Omega} \abs{\nabla u}^2\,dx
		+ n \int_{\partial \Omega} \Re \alpha \abs{u}^2\,d\sigma +2\Re \int_{\partial \Omega} \alpha x \cdot (u \nabla \bar{u}) \,d\sigma\\
		= -2\Im \lambda \Im \int_{\Omega}x \cdot u \nabla \bar{u}\, dx 
		+n\Re\int_{\Omega} f \bar{u}\, dx + 2\Re \int_{\Omega} f x \cdot \nabla \bar{u}\, dx\\ 
		+n\Re\int_{\partial \Omega} g \bar{u}\, d\sigma + 2\Re \int_{\partial \Omega} g x \cdot \nabla \bar{u}\, d\sigma.
	\end{multline}
\end{lemma}
\begin{remark}\label{rmk:justified}
Notice that since $u\in \widetilde{\mathcal{D}}_0$ 
and $\alpha \in L^\infty(\partial \Omega)$,
the integrals in \eqref{eq:first_Omega_const}--\eqref{eq:second_Omega} 
are well defined 
(in particular, the boundary terms make sense 
in virtue of Lemma~\ref{lemma:Neumann_3/2epsilon}).  Since $u, \nabla u$ and $f \in L^2(\Omega)$ it is immediate to see that $\abs{\nabla u}^2, \abs{u}^2, \frac{x}{\abs{x}} \cdot \bar{u}\nabla u$ and $f\bar{u}$ are integrable functions on $\Omega$ (the last two follow from the Cauchy-Schwarz inequality). Moreover, since $u$ and $g\in L^2(\partial \Omega)$ and $\alpha \in L^\infty(\partial \Omega),$ from the H\"older and the Cauchy-Schwarz inequalities it follows that $\Re\alpha \abs{u}^2, \Im \alpha \abs{u}^2$ and  $g\bar{u}$ are integrable functions on $\partial \Omega.$  Now notice that since $u\in \widetilde{\mathcal{D}}_0,$ in particular, $u$ is compactly supported and therefore $\abs{x}\in L^\infty(\supp u).$ This fact implies that a multiplication by $\abs{x}$ of integrable functions on the support of $u$ does not pose any further integrability issue, this ensures that $\abs{x}\abs{\nabla u}^2, x\cdot u \nabla \bar{u}, \abs{x} f \bar{u}, f x \cdot \nabla \bar{u}$ are integrable on $\Omega$ and $\Re \alpha \abs{x} \abs{u}^2, \Im \alpha \abs{x} \abs{u}^2, \abs{x} g \bar{u}, g x \cdot \nabla \bar{u}$ are integrable on $\partial \Omega.$ For the last one we have also used that $\nabla u \in L^2(\partial \Omega)$ (see Lemma~\ref{lemma:Neumann_3/2epsilon}). Finally, $\frac{~\abs{u}^2}{\abs{x}}$ is integrable on $\Omega$ in virtue of the weighted Hardy inequality~\eqref{eq:weighted_Hardy} and the fact that $\int_\Omega\abs{x}\abs{\nabla u}^2$ is finite.
\end{remark}
\begin{remark}\label{rmk:1d_identities}
The restriction to dimensions $n \geq 2$ in the lemma
is just because we did not find a unified way how to write 
the identities in all dimensions. 
The analogues of \eqref{eq:first_Omega_const}--\eqref{eq:second_Omega}
for the (somewhat special) case $n=1$ read
\begin{equation} 
\tag{\ref*{eq:first_Omega_const}'}
	\int_0^\infty \abs*{\frac{d}{dx} u}^2\, dx + \Re \alpha \abs{u}^2(0)= \Re\lambda \int_0^\infty \abs{u}^2\, dx + \Re \int_0^\infty f \bar{u}\, dx + \Re(g \bar{u}(0)),
\end{equation}
\begin{equation}
\tag{\ref*{eq:first_Omega_x}'}
\label{eq:negative_term}
	\int_0^\infty x \abs*{\frac{d}{dx}u}^2\, dx - \frac{1}{2}\abs{u}^2(0)= \Re \lambda \int_0^\infty x \abs{u}^2\, dx + \Re \int_0^\infty x f \bar{u}\, dx,
\end{equation}
\begin{equation}
\tag{\ref*{eq:first_Omega_im_const}'}
	\Im \alpha \abs{u}^2(0)= \Im \lambda \int_0^\infty \abs{u}^2\, dx + \Im \int_0^\infty f \bar{u}\, dx + \Im (g \bar{u}(0)), 
\end{equation}
\begin{equation}
\tag{\ref*{eq:first_Omega_im_x}'}
	\Im \int_0^\infty \bar{u}\frac{d}{dx}u\, dx= \Im \lambda \int_0^\infty x \abs{u}^2\, dx + \Im \int_0^\infty x f \bar{u}\, dx,
\end{equation}
\begin{equation}
\tag{\ref*{eq:second_Omega}'}
		2\int_0^\infty \abs*{\frac{d}{dx}u}^2\, dx + \Re \alpha \abs{u}^2(0)=-2\Im\lambda \Im \int_0^\infty x u \frac{d}{dx}\bar{u}\, dx + \Re \int_0^\infty f \bar{u}\, dx + 2\Re \int_0^\infty f x \frac{d}{dx}\bar{u}\, dx + \Re(g \bar{u}(0)).
\end{equation}
Notice also that if $n=1$ 
then $\alpha$ and $g$ are just complex numbers 
and Remark~\ref{rmk:justified} applies here as well.
\end{remark}

The proof of the previous lemma follows 
the standard technique of multipliers roughly described above.
The initial idea goes back to Morawetz~\cite{Morawetz_1968}
(see also~\cite{Morawetz-Ludwig_1968}), 
who employed a combination of various choices of multipliers 
to study the time-decay of the Klein-Gordon equation,
and the technique has been developed in several other contexts since. 
To name just a few among the most outstanding subsequent works exploiting this method, let us mention the seminal work~\cite{Perthame-Vega_1999} for the Helmholtz equation and some next generalisations~\cite{Barcelo-Fanelli-Ruiz-Vilela_2013,Barcelo-Vega-Zubeldia_2013,Cacciafesta-Ancona-Luca}. 
Concerning the Schr\"odinger equation we should quote~\cite{D'Ancona-Fanelli_2009,Cassano-Ancona_2012,FKV,FKV2} and for the Dirac equation~\cite{Boussaid-Ancona-Fanelli_2011,Cacciafesta_2011}. Moreover we mention~\cite{Cossetti_2017} for an adaptation of this method to an elasticity setting. 

In this section, we just sketch the main ideas behind the proof of 
Lemma~\ref{lemma:5identities}
and postpone the complete proof to Appendix~\ref{Appendix:5_identities}.

\begin{proof}[Sketch of the proof of Lemma~\ref{lemma:5identities}]
Equations~\eqref{eq:first_Omega_const} and~\eqref{eq:first_Omega_x} 
are obtained by choosing $v:=\varphi u$ in~\eqref{eq:weak_resolvent_gen},
taking the real part of the resulting identity, 
performing some integration by parts 
and at last choosing $\varphi(x):=1$ and $\varphi(x):=\abs{x}$, respectively.  

Equations~\eqref{eq:first_Omega_im_const} and~\eqref{eq:first_Omega_im_x} 
are derived by choosing $v:=\psi u$ in~\eqref{eq:weak_resolvent_gen},
taking the imaginary part of the resulting identity, 
performing some integration by parts and at last choosing 
$\psi(x):=1$ and $\psi(x):=\abs{x}$, respectively. 

Finally, choosing 
$
  v
  := [\Delta, \phi]u
  = \Delta \phi u + 2 \nabla \phi \cdot \nabla u
$ 
in~\eqref{eq:weak_resolvent_gen},
taking the real part of the resulting identity, 
performing some integration by parts 
and at last choosing $\phi=\abs{x}^2$ 
together with multiplying by~$1/2$ the resulting identity,
we obtain equation~\eqref{eq:second_Omega}. 
\end{proof}
\begin{remark}\label{Rem.Mourre}
We remark that the last identity~\eqref{eq:second_Omega} 
is closely related to the commutator theory \emph{\`a la} Mourre
of conjugate operators 
(see~\cite{Mourre} for the pioneering work and~\cite{ABG} for more recent developments)
and the virial theorem in quantum mechanics
(see~\cite{Weidmann_1967} for a first rigorous treatment 
and~\cite[Sec.~XIII.13]{RS4} for an overview and further references).
Indeed, with the choice above, one has $v = 4i A u$,
where $A := -\frac{i}{2}(x\cdot\nabla+\nabla\cdot x)$
is the usual dilation operator
(\emph{i.e.}\ the symmetrised version of the radial momentum in quantum mechanics).
The algebraic manipulations described above are related
to computing the formal commutator $i[-\Delta,A] = 2(-\Delta)$, which is positive,
and to taking into account the contributions of~$f$, $g$ and~$\alpha$.
A refined combination with the other identities of Lemma~\ref{lemma:5identities}
enables one to go beyond the standard self-adjoint setting 
and consider finer spectral properties,
see Remark~\ref{Rem.Eidus} below.
\end{remark}
%

\section{Proofs}\label{section:proofs}
This section is devoted to establish our main results
collected in Theorems~\ref{thm:self-adjoint}--\ref{thm:resolvent}.
As already underlined, their proofs share, as a starting point, 
the usage of the identities contained in Lemma~\ref{lemma:5identities}. 
Hence our first step consists in approximating solutions 
of~\eqref{eq:resolvent_gen} when $f=g=0$
(and in particular of~\eqref{eq:Schr_half_plane} when $f=0$) 
by a sequence of compactly supported functions.

\subsection{Approximation by compactly supported solutions}\label{Sec.approx}
The desired approximation is achieved by a usual ``horizontal cut-off''.
Let $\xi\colon [0, \infty) \to [0,1]$ be a smooth function 
such that 
\begin{equation*}
	\xi(r)= 
	\begin{cases}
		1 \quad &0\leq r\leq 1,\\
		0  &r\geq 2.
	\end{cases}
\end{equation*} 
Given a positive number~$R$,
we set $\xi_R(x):=\xi(\abs{x}R^{-1}).$ 
Then $\xi_R\colon \R^n \to [0,1]$ is such that
\begin{equation}\label{def:xiR}
  \xi_R = 1 \quad \text{in}\quad B_R(0), \qquad 
  \xi_R= 0 \quad \text{in}\quad \R^n\setminus B_{2R}(0), \qquad 
  \abs{\nabla \xi_R}\leq cR^{-1}, \qquad 
  \abs{\Delta \xi_R}\leq c R^{-2},
\end{equation}
where $B_R(0)$ stands for the open ball centered at the origin and with radius $R>0$ and $c> 1$ is a suitable constant independent of~$R$.
For any function $h\colon \R^n \to \C$ 
we then define the compactly supported approximating family of functions 
by setting
\begin{equation*}
	h_R:=\xi_R h.
\end{equation*}
If $u\in \widetilde{\mathcal{D}}$ is 
a weak solution to~\eqref{eq:resolvent_gen} with $g=0$, 
it is not difficult to show that $u_R \in \widetilde{\mathcal{D}}_0$ 
solves in a weak sense the following problem
\begin{equation}\label{eq:resolvent_approx}
	\begin{system}
		-\Delta u_R&=\lambda u_R + \widetilde{f}_R \quad& \text{in}\quad \Omega,
\\
		-(u_R)_{x_n} + \alpha u_R &= \widetilde{g}_R &\text{on}\quad \partial \Omega,
	\end{system}
\end{equation}
where 
\begin{equation}\label{eq:error_terms}
\widetilde{f}_R:= f_R -2\nabla u \cdot \nabla \xi_R - u \Delta \xi_R
\qquad \mbox{and} \qquad 
\widetilde{g}_R:= u \nabla \xi_R \cdot \eta.
\end{equation}
Hereafter $\eta$ denotes the outgoing unit normal 
to the boundary $\partial \Omega,$ namely $\eta=(0,\dots, 0,-1).$ 
Notice that~$u_R$ satisfies the hypotheses of Lemma~\ref{lemma:5identities}, 
therefore identities \eqref{eq:first_Omega_const}--\eqref{eq:second_Omega} 
are available for solutions to~\eqref{eq:resolvent_approx}.
 
The next easy result shows that the extra terms which originate from the introduction of the horizontal cut-off, measured with respect to a suitable topology, 
become negligible as~$R$ increases. 

\begin{lemma}\label{lemma:fg}
Given $u\in \widetilde{\mathcal{D}}$,
let $\widetilde{f}_R$ and $\widetilde{g}_R$ be as in~\eqref{eq:error_terms}. 
Then the estimates 
\begin{align*}
  \norm{\widetilde{f}_R}_{L^2(\Omega)} 
  &\leq \norm{f_R}_{L^2(\Omega)} + \varepsilon_1(R), 
  & 
  \norm{x \widetilde{f}_R}_{L^2(\Omega)}
  &\leq \norm{x f_R}_{L^2(\Omega)} + \varepsilon_2(R),
  \\
  \norm{\widetilde{g}_R}_{L^2(\partial \Omega)}
  &\leq \varepsilon_3(R), 
  &
  \norm{x \widetilde{g}_R}_{L^2(\partial \Omega)}
  &\leq \varepsilon_4(R),
\end{align*}
hold true, where 
\begin{equation}\label{limit_R}
\lim_{R\to \infty} \varepsilon_i(R)=0,\qquad i=1,2,3,4.
\end{equation}
(The last inequality is trivial if $n=1$.)
\end{lemma}
\begin{proof}
By~\eqref{def:xiR} we have 
\begin{align*}
\norm{\widetilde{f}_R}_{L^2(\Omega)}
&\leq \norm{f_R}_{L^2(\Omega)} 
+ 2 \left(\int_\Omega \abs{\nabla u}^2 \abs{\nabla \xi_R}^2\right)^{1/2} 
+ \left(\int_\Omega \abs{u}^2 \abs{\Delta \xi_R}^2\right)^{1/2}
\\
&\leq \norm{f_R}_{L^2(\Omega)} 
+ \frac{2c}{R} \left(\int_{\{R<\abs{x}<2R\}} \abs{\nabla u}^2\right)^{1/2} 
+ \frac{c}{R^2} \left(\int_{\{R<\abs{x}<2R\}} \abs{u}^2\right)^{1/2}.  
\end{align*}
Since $u\in L^2(\R^n)$ and $\nabla u\in \big[L^2(\R^n)\big]^n$, 
the last two terms of the right-hand side tend to zero as~$R$ goes to infinity.
Similarly,
\begin{align*}
\norm{x \widetilde{f}_R}_{L^2(\Omega)}
&\leq \norm{x f_R}_{L^2(\Omega)} 
+ 2\left (\int_{\Omega} \abs{x}^2 \abs{\nabla u}^2 \abs{\nabla \xi_R}^2\right)^{1/2} 
+ \left(\int_{\Omega} \abs{x}^2 \abs{u}^2 \abs{\Delta \xi_R}^2\right)^{1/2}
\\
&\leq \norm{x f_R}_{L^2(\Omega)} 
+ 4c \left(\int_{\{R< \abs{x}<2R\}} \abs{\nabla u}^2\right)^{1/2} 
+ \frac{2c}{R}\left(\int_{\{R<\abs{x}<2R\}} \abs{u}^2\right)^{1/2},
\end{align*}
and again the last two terms go to zero as $R$ approaches infinity.
Furthermore,
\begin{equation*}
	\norm{\widetilde{g}_R}_{L^2(\partial \Omega)}^2\leq \frac{c^2}{R^2} \int_{\partial \Omega} \abs{u}^2,
\end{equation*}
and being $u\in L^2(\partial \Omega),$ 
the right-hand side tends to zero as~$R$ tends to infinity.
Finally,
\begin{equation*} 
	\norm{x\widetilde{g}_R}_{L^2(\partial \Omega)}^2\leq \int_{\partial \Omega} \abs{x}^2 \abs{u}^2 \abs{\nabla \xi_R}^2\leq 4c^2 
\int_{\partial \Omega\, \cap\, \{R<\abs{x}<2R\}} \abs{u}^2,
\end{equation*}
which again goes to zero as $R$ tends to infinity because $u\in L^2(\partial \Omega).$
	\end{proof}

Now we are in a position to show 
the total absence of eigenvalues of the Robin Laplacian.
Let us start by considering the self-adjoint situation.

\subsection{Self-adjoint setting: Proof of Theorem~\ref{thm:self-adjoint}}
\label{Sec.virial}
Let $-\Delta_\alpha^\Omega$ be self-adjoint, 
which is equivalent to requiring that
the boundary function~$\alpha$ is real-valued.
Consequently, the spectrum of $-\Delta_\alpha^\Omega$ is real,
in particular any eigenvalue $\lambda\in \R$.
Let $u \in \mathcal{D}(-\Delta_\alpha^\Omega)$ be any solution
to the eigenvalue equation $-\Delta_\alpha^\Omega u = \lambda u$. 
The strategy of our proof is to show that,
under the hypotheses of Theorem~\ref{thm:self-adjoint},
$u$~is necessarily identically zero.
We split the proof into two cases: 
$\lambda \geq 0$ and $\lambda<0$.

\subsubsection*{$\bullet$ Case $\lambda \geq 0.$}
The solution~$u$ belongs to $\widetilde{\mathcal{D}}$
and solves~\eqref{eq:Schr_half_plane} in the weak sense.
Then $u_R \in \widetilde{\mathcal{D}}$ solves~\eqref{eq:resolvent_approx} 
with $f_R=0$. 
Taking the difference~\eqref{eq:second_Omega}$-$\eqref{eq:first_Omega_const}, 
one gets
\begin{multline*}
	\int_{\Omega} \abs{\nabla u_R}^2\, dx + \lambda \int_{\Omega} \abs{u_R}^2\, dx + (n-1) \int_{\partial \Omega} \alpha \abs{u_R}^2\, d\sigma
	+2 \int_{\partial \Omega}\alpha x \cdot \Re(u_R \nabla \bar{u}_R)\, d\sigma\\
	=(n-1)\int_{\Omega} \widetilde{f}_R \bar{u}_R\, dx +2\Re \int_\Omega \widetilde{f}_R x\cdot \nabla \bar{u}_R\, dx + (n-1)\int_{\partial \Omega} \widetilde{g}_R \bar{u}\, d\sigma +2\Re \int_{\partial \Omega} \widetilde{g}_R x\cdot \nabla \bar{u}_R\, d\sigma.
\end{multline*}
An integration by parts in the last term on the first line yields
the crucial identity
\begin{multline}\label{virial}
	\int_{\Omega} \abs{\nabla u_R}^2\, dx + \lambda \int_{\Omega} \abs{u_R}^2\, dx - \int_{\partial \Omega} (x \cdot \nabla \alpha) \abs{u_R}^2\, d\sigma\\
	=(n-1)\int_{\Omega} \widetilde{f}_R \bar{u}_R\, dx +2\Re \int_\Omega \widetilde{f}_R x\cdot \nabla \bar{u}_R\, dx\\ + (n-1)\int_{\partial \Omega} \widetilde{g}_R \bar{u}_R\, d\sigma +2\Re \int_{\partial \Omega} \widetilde{g}_R x\cdot \nabla \bar{u}_R\, d\sigma.
\end{multline}
Now we want to pass to the limit $R\to \infty$ in this identity.

We start by estimating the terms on the second line of~\eqref{virial}.
Let us set
\begin{equation*}
  I_1 := (n-1) \Re \int_{\Omega} \widetilde{f}_R \bar{u}_{R}
  \qquad \mbox{and} \qquad
  I_2 := 2 \Re \int_{\Omega} \widetilde{f}_R\, x\cdot \nabla \bar{u}_{R}.
\end{equation*}
Using the Cauchy-Schwarz inequality and the estimates in Lemma~\ref{lemma:fg} 
(here $f=0$), we have
\begin{equation*}
	\begin{split}
	\abs{I_1}&\leq(n-1)\norm{\widetilde{f}_R}_{L^2(\Omega)} \norm{u_{R}}_{L^2(\Omega)}\leq (n-1) \varepsilon_1(R) \norm{u_{R}}_{L^2(\Omega)},\\
	\abs{I_2}&\leq 2 \norm{x \widetilde{f}_R}_{L^2(\Omega)}\norm{\nabla u_{R}}_{L^2(\Omega)}\leq 2 \varepsilon_2(R) \norm{\nabla u_{R}}_{L^2(\Omega)}.	\end{split}
\end{equation*}
Consequently, by the dominated convergence theorem and~\eqref{limit_R},
we see that~$I_1$ and~$I_2$ vanish as $R \to \infty$. 

For the terms on the third line of~\eqref{virial}, we write
\begin{equation*}
  I_3 := (n-1)\int_{\partial \Omega} \widetilde{g}_R \bar{u}_R\, d\sigma 
  \qquad \mbox{and} \qquad
  I_4 := 2\Re \int_{\partial \Omega} \widetilde{g}_R x\cdot \nabla \bar{u}_R\, d\sigma.
\end{equation*}
Making again use of the Cauchy-Schwarz 
and of the estimates in Lemma~\ref{lemma:fg}, 
we obtain
\begin{equation*}
	\begin{split}
	\abs{I_3}&\leq (n-1) \norm{\widetilde{g}_R}_{L^2(\partial \Omega)} \norm{u_{R}}_{L^2(\partial \Omega)}\leq (n-1) \varepsilon_3(R) \norm{u_{R}}_{L^2(\partial \Omega)},\\
	\abs{I_4}&\leq 2 \norm{x \widetilde{g}_R}_{L^2(\partial \Omega)} \norm{\nabla u_{R}}_{L^2(\partial \Omega)}\leq 2 \varepsilon_4(R) \norm{\nabla u_{R}}_{L^2(\partial \Omega)}.
	\end{split}
\end{equation*}
Notice that $u \in \widetilde{\mathcal{D}}$ guarantees that 
the boundary traces $\gamma_D u$ and $\gamma_D \nabla u$
belong to $L^2(\partial \Omega)$.
Consequently, using again the dominated convergence theorem and~\eqref{limit_R},
we see that also~$I_3$ and~$I_4$ vanish as $R \to \infty$. 

Summing up, sending $R \to \infty$ in~\eqref{virial},
the right-hand side of the identity tends to zero.
Using additionally the dominated convergence theorem for 
the first two terms on the left-hand side 
and the monotone convergence theorem for the third term
(recall~\eqref{hp:repulsivity}),
we finally get the virial-type identity
\begin{equation*}
	\int_{\Omega} \abs{\nabla u}^2\, dx + \lambda \int_\Omega \abs{u}^2\, dx - \int_{\partial \Omega} (x\cdot \nabla \alpha)\abs{u}^2\, d\sigma=0.
\end{equation*}
Being $\lambda \geq 0$ 
and using hypothesis~\eqref{hp:repulsivity}, 
we obtain $\nabla u=0$. 
Consequently, $u=0$.

\subsubsection*{$\bullet$ Case $\lambda < 0.$}
Identity~\eqref{eq:first_Omega_const} reads
\begin{equation*}
	\int_{\Omega} \abs{\nabla u_R}^2 \, dx + \int_{\partial \Omega} \alpha \abs{u_R}^2 \, d\sigma= \lambda \int_{\Omega} \abs{u_R}^2\, dx + \Re \int_{\Omega} \widetilde{f}_R \bar{u}_R\, dx + \Re \int_{\partial \Omega} \widetilde{g}_R \bar{u}_R\, d\sigma.
\end{equation*}
Sending $R\to \infty$ 
and estimating the terms involving $\widetilde{f}_R$ and $\widetilde{g}_R$ 
as~$I_1$ and~$I_3$ above,
the dominated convergence theorem yields
\begin{equation*}
	\int_{\Omega} \abs{\nabla u}^2 \, dx + \int_{\partial \Omega} \alpha \abs{u}^2 \, d\sigma= \lambda \int_{\Omega} \abs{u}^2\, dx.
\end{equation*}
(In fact, this identity is just $h_\alpha[u]=\lambda \|u\|_{L^2(\Omega)}^2$,
the passage through the compactly supported regularisation 
was not necessary in this case.)
Being $\lambda < 0$ and using the hypothesis~\eqref{hp:self-adjoint}, 
we again obtain $u=0.$ 

\medskip

This concludes the proof of Theorem~\ref{thm:self-adjoint}.
\qed

\medskip
Now we turn to the non-self-adjoint situation.


\subsection{Non-self-adjoint setting: Auxiliary results}

Before moving on in the proof of 
Theorems~\ref{thm:non_self-adjoint} and~\ref{thm:non_self-adjoint_h}, 
we first present some preliminary results that will be used in the sequel.

To begin with, we recall the following Hardy inequalities on the half-space.

\begin{lemma}\label{lemma:Hardy_inequalities}
Let $n\geq 3$. 
For every $\psi \in H^1(\R^n)$, one has
\begin{align}
\int_{\Omega} \frac{\abs{\psi}^2}{\abs{x}^2}\, dx 
&\leq \frac{4}{(n-2)^2}\int_{\Omega} \abs{\nabla \psi}^2\, dx, 
\label{eq:Hardy}
\\
\int_{\Omega} \frac{~\abs{\psi}^2}{\abs{x}}\, dx 
&\leq \frac{4}{(n-1)^2}\int_{\Omega} \abs{x} \abs{\nabla \psi}^2\, dx. 
\label{eq:weighted_Hardy}
\end{align}
(Inequality~\eqref{eq:weighted_Hardy} holds also for $n=2$.)
\end{lemma}
\begin{proof}
The inequalities are well known in the case $\Omega = \R^n$.
For the sake of completeness, we show that the standard proof
extends to the present situation of~$\Omega$ being the half-space.
Noticing that if $n\geq 2$ then $H^1(\R^n)=H^1_0(\R^n\setminus \{0\})$,
it is enough to prove the inequalities 
for functions $\psi\in C^{\infty}_0(\R^n\setminus\{0\})$.
Then one has
	\begin{equation*}
		\begin{split}
		\abs*{\nabla \psi + \frac{n-2}{2}\frac{\psi}{\abs{x}}\frac{x}{\abs{x}}}^2&=\abs{\nabla \psi}^2 + \frac{(n-2)^2}{4}\frac{\abs{\psi}^2}{\abs{x}^2} + \frac{n-2}{2}\, 2 \Re(\bar{\psi} \nabla \psi)\cdot \frac{x}{~\abs{x}^2}\\
		&=\abs{\nabla \psi}^2 + \frac{(n-2)^2}{4} \frac{\abs{\psi}^2}{\abs{x}^2} + \frac{n-2}{2}\, \nabla \abs{\psi}^2 \cdot \frac{x}{~\abs{x}^2}.
		\end{split}
	\end{equation*}
Reintegrating the previous identity over~$\Omega$ 
and integrating by parts, we get
	\begin{equation*}
		\int_{\Omega} \abs*{\nabla \psi + \frac{n-2}{2}\frac{\psi}{\abs{x}}\frac{x}{\abs{x}}}^2\, dx=\int_\Omega \abs{\nabla \psi}^2\, dx - \frac{(n-2)^2}{4}  \int_\Omega \frac{\abs{\psi}^2}{\abs{x}^2}\, dx  + \frac{(n-2)}{2} \int_{\partial \Omega}  \frac{\abs{\psi}^2}{\abs{x}^2} x \cdot \eta\, d\sigma.
	\end{equation*}
Since the left-hand side is positive 
and observing that $x\cdot \eta=0$ on $\partial \Omega$,
we arrive at~\eqref{eq:Hardy}. 
Inequality~\eqref{eq:weighted_Hardy} can be proved similarly.
\end{proof}

Given any complex number~$\lambda$ 
and complex-valued functions 
$\alpha \in L^\infty(\partial\Omega)$ and $f \in L^2(\Omega)$, 
let us now consider solutions $u \in \mathcal{D}(-\Delta_\alpha^\Omega)$
to the resolvent equation $(-\Delta_\alpha^\Omega-\lambda)u=f$.
The function~$u$ solves the boundary-value problem
\begin{equation}\label{eq:resolvent}
	\begin{system}
		-\Delta u&= \lambda u +f \quad& \text{in}\quad \Omega,\\
		-u_{x_{n}} + \alpha u&=0 &\text{on}\quad \partial \Omega,
	\end{system}
\end{equation}
in the weak sense of~\eqref{eq:weak_resolvent_gen} with $g=0$.	
The following lemma represents the crucial tool in this section.
\begin{lemma}\label{crucial_identity}
Let $n\geq 2$
and assume that $u\in \widetilde{\mathcal{D}}$ 
is a weak solution to~\eqref{eq:resolvent}
with $0<\abs{\Im \lambda}\leq \Re \lambda.$ 
Then $\abs{x}\abs{\nabla u^-}^2\in L^1(\Omega)$ and the inequality
	\begin{multline}\label{eq:crucial_identity}
		\int_{\Omega} \abs{\nabla u^-}^2\, dx + \frac{n-3}{n-1}(\Re \lambda)^{-1/2} \abs{\Im \lambda} \int_{\Omega} \abs{x} \abs{\nabla u^-}^2\, dx\\
		+(n-1)\int_{\partial \Omega} \Re \alpha \abs{u}^2\, d\sigma + 2 \Re \int_{\partial \Omega} x \, \alpha u^- \cdot \overline{\nabla u^-}\, d\sigma + (\Re \lambda)^{-{1/2}} \abs{\Im \lambda} \int_{\partial \Omega} \abs{x} \Re \alpha \abs{u}^2\, d\sigma\\
\leq		
(n-1) \Re \int_{\Omega} f \bar{u}\, dx + 2 \Re \int_{\Omega} x f^-\cdot \overline{\nabla u^-}\, dx + (\Re \lambda)^{-{1/2}} \abs{\Im \lambda} \Re \int_{\Omega} f \abs{x} \bar{u}\, dx
	\end{multline}
holds true with $u^-$ being defined as in~\eqref{eq:u^-} 
and $f^-(x):=e^{-i(\Re \lambda)^{1/2} \sgn(\Im \lambda) \abs{x}} f(x).$
\end{lemma}

\begin{proof}
To begin with, we consider the compactly supported approximating sequence~$u_R$ 
as introduced in Section~\ref{Sec.approx}.
Then the proof follows after a precise combination of the identities contained in Lemma~\ref{lemma:5identities} once they are re-stated in terms of $u_R, \widetilde{f}_R$ and $\widetilde{g}_R,$ where $\widetilde{f}_R$ and $\widetilde{g}_R$ are defined in~\eqref{eq:error_terms}. 
 
Let us start our algebraic manipulations by taking the sum 
\begin{equation*}
-\,\eqref{eq:first_Omega_const}
\,-	\, 2(\Re\lambda)^{1/2} \sgn(\Im\lambda) \eqref{eq:first_Omega_im_x} 
\, + \, \eqref{eq:second_Omega}.
\end{equation*}
This gives
\begin{multline}\label{eq:intermediate}
	\hspace{2.5cm}\int_{\Omega} \abs{\nabla u_R}^2\, dx -2 (\Re\lambda)^\frac{1}{2} \sgn(\Im\lambda) \Im \int_{\Omega} \frac{x}{\abs{x}} \cdot (\bar{u}_R\,\nabla u_R)\, dx + \Re \lambda \int_{\Omega} \abs{u_R}^2\, dx\\
	+2(\Re \lambda)^{1/2} \abs{\Im \lambda} \int_{\Omega} \abs{x} \abs{u_R}^2\, dx + 2 \Im\lambda \Im \int_{\Omega} x \cdot (u_R\, \nabla \bar{u}_R)\, dx\\
	+(n-1)\int_{\partial \Omega} \Re \alpha \abs{u_R}^2\, d\sigma + 2 \Re \int_{\partial \Omega} (x \alpha u)\cdot \nabla \bar{u}_R\, d\sigma
	-2(\Re \lambda)^{{1/2}} \sgn(\Im \lambda) \int_{\partial \Omega} \abs{x} \Im \alpha \abs{u_R}^2\, d\sigma\\
	= (n-1)\Re \int_{\Omega} \widetilde{f}_R \bar{u}_R\, dx + 2 \Re \int_{\Omega} \widetilde{f}_R x\cdot \nabla \bar{u}_R\, dx - 2(\Re \lambda)^{{1/2}} \sgn(\Im\lambda) \Im \int_{\Omega} \abs{x} \widetilde{f}_R \bar{u}_R\, dx\\
	(n-1)\Re \int_{\partial \Omega} \widetilde{g}_R \bar{u}_R\, d\sigma + 2 \Re \int_{\partial \Omega} \widetilde{g}_R x\cdot \nabla \bar{u}_R\, d\sigma - 2(\Re \lambda)^{{1/2}} \sgn(\Im\lambda) \Im \int_{\partial \Omega} \abs{x} \widetilde{g}_R \bar{u}_R\, d\sigma. 
\end{multline}
Recalling the definition~\eqref{eq:u^-}, one observes that
\begin{equation}\label{nabla_u^-}
	\nabla u_R^-(x)= e^{-i (\Re \lambda)^{1/2} \sgn(\Im \lambda) \abs{x}} \left(\nabla u_R - i (\Re \lambda)^\frac{1}{2} \sgn(\Im \lambda) \frac{x}{\abs{x}} u_R(x)\right),
\end{equation}
and therefore
\begin{equation}\label{eq:square_nabla_u^-}
	\abs{\nabla u_R^-}^2= \abs{\nabla u_R}^2 + \Re \lambda \abs{u_R}^2 - 2 (\Re \lambda)^{1/2} \sgn(\Im\lambda) \Im \left( \frac{x}{\abs{x}} \bar{u}_R\nabla u_R\right). 
\end{equation}
Reintegrating~\eqref{eq:square_nabla_u^-} over~$\Omega$,
we obtain
\begin{equation}\label{eq:nabla_u^-}
	\int_{\Omega} \abs{\nabla u_R}^2\, dx -2 (\Re\lambda)^{1/2} \sgn(\Im\lambda) \Im \int_{\Omega} \frac{x}{\abs{x}} \cdot (\bar{u}_R\,\nabla u_R)\, dx + \Re \lambda \int_{\Omega} \abs{u_R}^2\, dx=\int_{\Omega} \abs{\nabla u_R^-}^2\, dx.
\end{equation}
Adding equation~\eqref{eq:first_Omega_x} 
multiplied by $(\Re \lambda)^{-{1/2}} \abs{\Im \lambda}$ to~\eqref{eq:intermediate}, plugging~\eqref{eq:nabla_u^-} and using again~\eqref{eq:square_nabla_u^-}, 
we get
\begin{multline}\label{eq:not_yet_minus}
	\qquad \qquad \int_{\Omega} \abs{\nabla u_R^-}^2\, dx + (\Re\lambda)^{-{1/2}} \abs{\Im \lambda} \int_{\Omega} \abs{x} \abs{\nabla u_R^-}^2\, dx -\frac{n-1}{2} (\Re \lambda)^{-{1/2}} \abs{\Im \lambda} \int_{\Omega} \frac{~\abs{u_R}^2}{\abs{x}}\, dx\\
	+(n-1)\int_{\partial \Omega} \Re \alpha \abs{u_R}^2\, d\sigma
	+(\Re \lambda)^{-{1/2}} \abs{\Im \lambda} \int_{\partial \Omega} \abs{x} \Re \alpha \abs{u_R}^2\, d\sigma
	\\
	+ 2 \Re \int_{\partial \Omega} (x\, \alpha\, u_R)\cdot \nabla\, \bar{u}_R\, d\sigma
	-2(\Re \lambda)^{{1/2}} \sgn(\Im \lambda) \Im \int_{\partial \Omega} \abs{x} \alpha \abs{u_R}^2\, d\sigma\\
	= (n-1)\Re \int_{\Omega} \widetilde{f}_R \bar{u}_R\, dx 
	+(\Re \lambda)^{-{1/2}} \abs{\Im \lambda} \Re \int_{\Omega} \widetilde{f}_R \abs{x} \bar{u}_R\, dx\\
	+ 2 \Re \int_{\Omega} \widetilde{f}_R x\cdot \overline{\nabla u_R}\, dx -2(\Re \lambda)^{1/2}\sgn(\Im\lambda) \Im \int_{\Omega} \abs{x} \widetilde{f}_R \bar{u}_R\, dx\\
		+(n-1)\Re \int_{\partial \Omega} \widetilde{g}_R \bar{u}_R\, d\sigma
		+(\Re \lambda)^{-{1/2}} \abs{\Im \lambda} \Re \int_{\partial \Omega} \widetilde{g}_R \abs{x} \bar{u}_R\, d\sigma\\
	+ 2 \Re \int_{\partial \Omega} \widetilde{g}_R x\cdot \overline{\nabla u_R}\, d\sigma 
	-2(\Re \lambda)^{1/2} \sgn(\Im \lambda) \Im \int_{\partial \Omega} \abs{x} \widetilde{g}_R \bar{u}_R\, d\sigma. \hspace{2.5cm}
\end{multline}
In order to estimate the first line of~\eqref{eq:not_yet_minus}, we use the weighted Hardy inequality~\eqref{eq:weighted_Hardy}. Then from~\eqref{eq:weighted_Hardy} and the fact that $\abs{u_R}=\abs{u_R^-}$ we get
\begin{multline}\label{l-h-s}
	\hspace{2cm}\int_{\Omega} \abs{\nabla u_R^-}^2\, dx +  (\Re \lambda)^{-\frac{1}{2}}\abs{\Im \lambda} \int_{\Omega} \abs{x} \abs{\nabla u_R^-}^2\, dx - \frac{(n-1)}{2}(\Re \lambda)^{-\frac{1}{2}}  \abs{\Im \lambda} \int_{\Omega} \frac{\abs{u_R}^2}{\abs{x}}\, dx\\ \geq \int_{\Omega} \abs{\nabla u_R^-}^2\, dx + \frac{n-3}{n-1}  (\Re \lambda)^{-\frac{1}{2}} \abs{\Im \lambda} \int_{\Omega} \abs{x} \abs{\nabla u_R^-}^2\, dx. 
\end{multline}
In order to deal with the third line of~\eqref{eq:not_yet_minus}, 
we recall~\eqref{nabla_u^-} and write
\begin{equation}\label{eq:J_potential}
  2 \Re \int_{\partial \Omega} (x\, \alpha\, u_R)\cdot \nabla\, \bar{u}_R\, d\sigma
	-2(\Re \lambda)^{{1/2}} \sgn(\Im \lambda) \Im \int_{\partial \Omega} \abs{x} \alpha \abs{u_R}^2\, d\sigma
  = 2 \Re \int_{\partial \Omega} x \alpha u_R^-\cdot \overline{\nabla u_R^-}\, d\sigma.
\end{equation}
Plugging~\eqref{l-h-s} and~\eqref{eq:J_potential} in~\eqref{eq:not_yet_minus} and using again~\eqref{nabla_u^-} also to treat the terms involving $\widetilde{f}_R$ in the last but three line of~\eqref{eq:not_yet_minus} and the terms involving $\widetilde{g}_R$ in the last line of~\eqref{eq:not_yet_minus}, one gets 
\begin{multline}\label{eq:last_but_one}
	\int_{\Omega} \abs{\nabla u_R^-}^2\, dx + \frac{n-3}{n-1} (\Re \lambda)^{-{1/2}} \abs{\Im \lambda} \int_{\Omega} \abs{x} \abs{\nabla u_R^-}^2\, dx\\
	+(n-1)\int_{\partial \Omega} \Re \alpha \abs{u_R}^2\, d\sigma
	+ 2 \Re \int_{\partial \Omega} x \alpha u_R^- \cdot \overline{\nabla u_R^-}\, d\sigma
	+(\Re \lambda)^{-{1/2}} \abs{\Im \lambda} \int_{\partial \Omega} \abs{x} \Re \alpha \abs{u_R}^2\, d\sigma\\
	\leq (n-1)\Re \int_{\Omega} \widetilde{f}_R \bar{u}_R\, dx 
	+ 2 \Re \int_{\Omega} \widetilde{f}_R^- x\cdot \overline{\nabla u_R^-}\, dx
		+(\Re \lambda)^{-{1/2}} \abs{\Im \lambda} \Re \int_{\Omega} \widetilde{f}_R \abs{x} \bar{u}_R\, dx\\
		(n-1)\Re \int_{\partial \Omega} \widetilde{g}_R \bar{u}_R\, d\sigma 
	+ 2 \Re \int_{\partial \Omega} \widetilde{g}_R^- x\cdot \overline{\nabla u_R^-}\, d\sigma
		+(\Re \lambda)^{-{1/2}} \abs{\Im \lambda} \Re \int_{\partial \Omega} \widetilde{g}_R \abs{x} \bar{u}_R\, d\sigma.
\end{multline}
Estimating the right-hand side of~\eqref{eq:last_but_one}
by means of the Cauchy-Schwarz inequality and Lemma~\ref{lemma:fg}, 
and letting~$R$ go to infinity with help of the monotone and dominated convergence theorems, one gets the thesis. 
\end{proof}
\begin{remark}
Using the identities of Remark~\ref{rmk:1d_identities}, 
one can easily check that if $n=1$,
then~\eqref{eq:not_yet_minus}
reads (after using~\eqref{nabla_u^-} for the term involving $\widetilde{f}_R$ and $\widetilde{g}_R$ and passing to the limit $R \to \infty$) 
\begin{multline*}
\int_0^\infty \abs*{\frac{d}{dx}u^-}^2\, dx 
+ (\Re \lambda)^{-{1/2}} \abs{\Im \lambda}\int_0^\infty x \abs*{\frac{d}{dx}u^-}^2\, dx 
-\frac{1}{2}(\Re\lambda)^{-{1/2}} \abs{\Im \lambda} \abs{u}^2(0)
\\
=2\Re \int_0^\infty x f^- \overline{\frac{d}{dx}u^-}\, dx
+ (\Re \lambda)^{-{1/2}}\abs{\Im\lambda} \Re \int_0^\infty f x \bar{u}.
		\end{multline*}
Hence, all the boundary terms appearing 
in~\eqref{eq:not_yet_minus} disappear,
but there shows up a new boundary term, namely 
$-\frac{1}{2} (\Re \lambda)^{-{1/2}}\abs{\Im \lambda}\abs{u}^2(0)$,
coming from~\ref{eq:negative_term}. 
The negative sign of this extra term does not allow 
to get in $n=1$ the same result as the one stated in 
Theorem~\ref{thm:non_self-adjoint} obtained in higher dimensions $n \geq 3$.
\end{remark}

\begin{remark}\label{Rem.Eidus}
Let us underline that in the non-self-adjoint situation the introduction of the auxiliary function~$u^-$ is crucial in order to guarantee the absence of eigenvalues also inside the cone $\{\lambda \in \C \ | \, \abs{\Im \lambda}<\Re \lambda\}$.
		
As far as we know, the first appearance of~$u^-$ in the literature goes back to the work by Eidus~\cite{Eidus_1962} in 1962 concerning the electromagnetic Helmholtz equation

		\begin{equation}\label{eq:Eidus}
			(\nabla + iA(x))^2u + V(x)u + \lambda u=f(x),\quad x\in \R^3,
		\end{equation}
where $V:\R^n\to\R$ and $A:\R^n\to\R^n$.
		Here the author, studying deeply the equation~\eqref{eq:Eidus} 
in the regime $\Im \lambda=0$ where no uniqueness is guaranteed in general, 
singled out among the solutions the one that is the most interesting from a physical point of view. More specifically, he looked for a solution satisfying the so-called Sommerfeld radiation conditions
		\begin{equation}\label{rad_cond}			
			\lim_{r\to \infty} \int_{\{\abs{x}=r\}} \abs*{\nabla u - i (\Re \lambda)^{1/2} \frac{x}{\abs{x}} u}^2\, d\sigma=0,
		\end{equation}
		whose physical meaning is the absence of any sources of oscillation at infinity.
Notice that~\eqref{rad_cond} is nothing but requiring that
		\begin{equation*}			
			\lim_{r\to \infty} \int_{\{\abs{x}=r\}} \abs*{\nabla u^-}^2\, d\sigma=0.
		\end{equation*}
In 1972, aligned with Eidus, Ikebe and Saito devoted 
the greater part of their work~\cite{Ikebe-Saito_1972} 
to the justification of the limit absorption principle for~\eqref{eq:Eidus}, 
which followed from the proof of suitable \emph{a priori} estimates 
for the selected solution~$u^-.$  The main step forward of this work was that the authors, in order to get the aforementioned \emph{a priori} estimates for~$u^-,$  recognised the need to \emph{combine} in a suitable algebra identities resulting from \emph{both} the choice of ``symmetric'' and ``anti-symmetric'' multipliers 
(in the spirit of those used in the proof of Lemma~\ref{lemma:5identities}). 
This novelty in the method of multipliers paved 
the way for what it grew up to become  a baseline technique in the matter of resolvent estimates and which inspired a considerable literature on related topics 
in partial differential equations afterward.
Unfortunately, this improvement upon the commutator method 
(\emph{cf.}~Remark~\ref{Rem.Mourre})
does not seem to be well known in the spectral-theoretic community.
\end{remark}

We also give the following variant of Lemma~\ref{crucial_identity}.
\begin{lemma}\label{crucial_identity_h}
Let $n\geq 2$ and assume $u\in \widetilde{\mathcal{D}}$ 
is a weak solution to~\eqref{eq:resolvent} with $0<\abs{\Im \lambda}\leq \Re \lambda.$ Then $\abs{x} \abs{\nabla u^-}^2\in L^1(\Omega)$, 
$x\cdot \nabla \Re \alpha \abs{u}^2 \in L^1(\partial \Omega)$ 
and the inequality
		\begin{multline}\label{eq:crucial_identity_h}
		\int_{\Omega} \abs{\nabla u^-}^2\, dx 
+ \frac{n-3}{n-1}(\Re \lambda)^{-{1/2}} \abs{\Im \lambda} \int_{\Omega} \abs{x} \abs{\nabla u^-}^2\, dx
\\
		-\int_{\partial \Omega} x\cdot \nabla \Re \alpha \abs{u}^2\, d\sigma - 2 \Im \int_{\partial \Omega} x \, \Im \alpha u^- \cdot \overline{\nabla u^-}\, d\sigma + (\Re \lambda)^{-{1/2}} \abs{\Im \lambda} \int_{\partial \Omega} \abs{x} \Re \alpha \abs{u}^2\, d\sigma\\
		\leq (n-1) \Re \int_{\Omega} f \bar{u}\, dx + 2 \Re \int_{\Omega} x f^-\cdot \overline{\nabla u^-}\, dx + (\Re \lambda)^{-{1/2}} \abs{\Im \lambda} \Re \int_{\Omega} f \abs{x} \bar{u}\, dx
	\end{multline}
	holds true.
\end{lemma}
\begin{proof}
As a starting point we consider 
inequality~\eqref{eq:last_but_one} 
in the proof of Lemma~\ref{crucial_identity}. 
Defining
\begin{equation*}
  J := 
  2\Re \int_{\partial \Omega} (x \alpha u_R^-)\cdot \overline{\nabla u_R^-}\, d\sigma,
\end{equation*}
which appears on the third line 
of~\eqref{eq:last_but_one},
we write $J = J_1 + J_2$ with
\begin{equation*}
  J_1 := 2 \int_{\partial \Omega} (x\, \Re \alpha) \cdot \Re (u_R^-\, \overline{\nabla\, u_R^-})\, d\sigma 
\qquad \mbox{and} \qquad
  J_2 := - 2\Im \int_{\partial \Omega} x\, \Im \alpha\, u_R^- \cdot \overline{\nabla\, u_R^-}\, d\sigma. 
\end{equation*}
Integrating by parts in $J_1$ gives
\begin{equation*}
		J =-(n-1)\int_{\partial \Omega} \Re \alpha \abs{u_R}^2 \, d\sigma -\int_{\partial \Omega} x \cdot \nabla \Re \alpha\, \abs{u_R}^2\, d\sigma - 2\Im \int_{\partial \Omega} x\, \Im \alpha\, u_R^- \cdot \overline{\nabla\, u_R^-}\, d\sigma.
\end{equation*}
Using the latter in~\eqref{eq:last_but_one}
and estimating the right-hand side of the resulting identity 
by means of the Cauchy-Schwarz inequality and Lemma~\ref{lemma:fg}, 
then~\eqref{eq:crucial_identity_h} follows after passing 
to the limit $R\to \infty$ with help of
the monotone and dominated convergence theorems.
\end{proof}

The following homogeneous trace-type result 
will be also useful in the proof of 
Theorems~\ref{thm:non_self-adjoint} and~\ref{thm:non_self-adjoint_h}.
\begin{lemma}
Let $n\geq 1$ and $u\in H^1(\Omega)$.
Then the following estimate
		\begin{equation}\label{eq:homogeneous_trace}
			\norm{u(0)}_{\dot{H}^{1/2}(\R^{n-1})}\leq \norm{\nabla u}_{L^2(\Omega)}
		\end{equation}
		holds true,
where $u(0)$ 
denotes the trace of~$u$ on the boundary~$\partial\Omega$ and $\norm{\cdot}_{\dot{H}^{1/2}(\R^{n-1})}$ is defined as in~\eqref{eq:homog-norm}.
\end{lemma}
\begin{proof}
		In order to prove~\eqref{eq:homogeneous_trace} we use the so-called \emph{harmonic extension argument} introduced by Caffarelli and Silvestre in their pioneering work~\cite{Caffarelli-Silvestre_2007} related to trace-type results in connection with fractional Sobolev
spaces. 
More specifically, given $f\colon \R^{n-1}\to \C$, 
as a particular case of~\cite[Eq.~(3.7) fixing $a=0$ there]{Caffarelli-Silvestre_2007}, 
one easily gets the following identity
\begin{equation*}
\norm{f}_{\dot{H}^{1/2}(\R^{n-1})}
= \int_{\R^{n-1}}\int_0^\infty \abs{\nabla U}^2\,dx_n\,dx',
\end{equation*}
where $U\in H^1(\Omega)$ is the (unique) harmonic extension of $f$ to the upper half-space, 
\emph{i.e.} $U\colon \Omega \to \C$ is such that
\begin{equation*}
	\begin{system}
\Delta U&=0 \qquad &\text{for}\quad (x',x_n)\in \R^{n-1} \times (0,\infty),
\\
U(x',0)&=f(x') &\text{for}\quad x'\in \R^{n-1}.
	\end{system}
\end{equation*} 
Observe that being~$U$ a harmonic function, 
it minimises the Dirichlet functional 
\begin{equation*}
  J[F]:=\int_{\R^{n-1}} \int_0^\infty \abs{\nabla F}^2\, dx_n\, dx',
\end{equation*}
over the set of function $F\colon \Omega \to \C$ 
such that the trace of~$F$ on the boundary, denoted by $F(0),$ 
satisfies $F(0)=f.$ As a consequence, one has
\begin{equation*}
	\norm{f}_{\dot{H}^{1/2}(\R^{n-1})}\leq \int_{\R^{n-1}} \int_0^\infty \abs{\nabla F}^2\, dx_n\, dx'.
\end{equation*}
Finally, taking $f:=u(0)$ and $F:=u$ in the previous inequality yields the thesis. 
\end{proof}

In passing, we stress that even though the original approach introduced by Caffarelli and Silvestre in~\cite{Caffarelli-Silvestre_2007} considered the extension to the half-space only, their idea has been generalised to cover different domains, see, for instance,~\cite{Capella-Davila-Dupaigne-Sire_2011} and~\cite{Stinga-Torrea_2010} for an adaptation to semi-infinite cylinder.

\subsection{Non-self-adjoint setting: Proof of Theorem~\ref{thm:non_self-adjoint}}
\label{Sec.nsa1} 
Let $u \in \mathcal{D}(-\Delta_\alpha^\Omega)$ be any solution
to the eigenvalue equation $-\Delta_\alpha^\Omega u = \lambda u$. 
The strategy is again to show that,
under the hypotheses of Theorem~\ref{thm:self-adjoint},
$u$~is necessarily identically zero.
Similarly to the self-adjoint setting, 
we split the proof into two cases: 
$\abs{\Im \lambda}\leq \Re \lambda$ and $\abs{\Im \lambda}> \Re \lambda.$  

\subsubsection*{$\bullet$ Case $\abs{\Im \lambda}\leq \Re \lambda.$}
As a starting point we consider inequality~\eqref{eq:crucial_identity} 
in Lemma~\ref{crucial_identity} with $f=0$, which reads as follows:
	\begin{multline}\label{eq:crucial_identity_f0}
		\int_{\Omega} \abs{\nabla u^-}^2\, dx + \frac{n-3}{n-1}(\Re \lambda)^{-{1/2}} \abs{\Im \lambda} \int_{\Omega} \abs{x} \abs{\nabla u^-}^2\, dx\\
		+(n-1)\int_{\partial \Omega} \Re \alpha \abs{u}^2\, d\sigma + 2 \Re \int_{\partial \Omega} (x \, \alpha u^-) \cdot \overline{\nabla u^-}\, d\sigma \\
		+ (\Re \lambda)^{-{1/2}} \abs{\Im \lambda} \int_{\partial \Omega} \abs{x} \Re \alpha \abs{u}^2\, d\sigma \leq 0.
	\end{multline}
We want to estimate the following term in~\eqref{eq:crucial_identity_f0}:
	\begin{equation*}
		I:=2 \Re \int_{\partial \Omega}(x\, \alpha\, u^-)\cdot \overline{\nabla u^-}\, d\sigma
=2\Re \int_{\R^{n-1}} (x'\, \alpha\, u^-(0))\cdot \overline{\nabla_{\!x'} u^-(0)}\, dx'.
	\end{equation*}
To this aim,		
we introduce the sesquilinear form
	\begin{equation*}
		T(f,g):=\int_{\R^{n-1}} (x'\, \alpha\, f)\cdot \nabla_{\!x'}\,\bar{g}\, dx',
	\end{equation*}
with suitable functions $f,g:\R^{n-1} \to \C$.
By virtue of the Plancherel theorem and the Cauchy-Schwarz inequality, we have
	\begin{equation}\label{eq:preliminary}
		\begin{split}
\abs{T(f,g)}
&\leq \abs{\big\langle x'\, \alpha\, f, \nabla_{x'}\, g \big\rangle_{L^2(\R^{n-1})}}
\leq \sum_{j=1}^{n-1} \big\langle \abs{\xi'}^{1/2}\,\abs{\mathcal{F}_{x'}\,(x_j\, \alpha\, f)}, \abs{\xi'}^{1/2}\,\abs{\mathcal{F}_{x'}\,g} \big\rangle_{L^2(\R^{n-1})}
\\
&\leq \norm{x'\, \alpha\, f}_{\dot H^{{1/2}}(\R^{n-1})} \norm{g}_{\dot H^{{1/2}}(\R^{n-1})} ,
		\end{split}
	\end{equation}
where the inner product $\langle \cdot, \cdot \rangle_{L^2(\R^{n-1})}$ 
is assumed to be conjugate linear in the second argument, 
$\mathcal{F}_{x'}$ denotes the Fourier transform 
with respect to the $(n-1)$-dimensional variable~$x'$,
$\xi'$ stands for the dual variable and
$$
  \norm{x'\, \alpha\, f}_{\dot H^{{1/2}}(\R^{n-1})}
  := \sum_{j=1}^{n-1} \norm{x_j\, \alpha\, f}_{\dot H^{{1/2}}(\R^{n-1})}.
$$
Applying~\eqref{eq:preliminary} to $I=2\Re T(u^-(0),u^-(0))$, 
we obtain the preliminary estimate
	\begin{equation}\label{eq:term_I}
		\abs{I}\leq 2 \abs{T(u^-(0),u^-(0))}\leq 2\norm{x'\, \alpha\, u^-(0)}_{\dot H^{{1/2}}(\R^{n-1})} \norm{u^-(0)}_{\dot H^{\frac{1}{2}}(\R^{n-1})}.
	\end{equation}
Let us now consider the term 
$\norm{h\, u^-(0)}_{\dot H^{{1/2}}(\R^{n-1})}$,
where we abbreviate $h:=x'\, \alpha$.
Using the fractional Leibniz rule~\eqref{eq:chain_rule} and the Sobolev embedding~\eqref{eq:Sobolev_inequality}, one gets
	\begin{equation*}
		\begin{split}
			\norm{h\, u^-(0)}_{\dot H^{{1/2}}(\R^{n-1})}&= \norm{(-\Delta)^\frac{1}{4}(h u^-(0))}_{L^2(\R^{n-1})}\\
			&\leq C^*\big[ \norm{h}_{L^\infty(\R^{n-1})} \norm{u^-(0)}_{\dot H^{{1/2}}(\R^{n-1})} + \norm{(-\Delta)^\frac{1}{4} h}_{L^{2(n-1)}(\R^{n-1})} \norm{u^-(0)}_{L^{2^\ast}(\R^{n-1})}\big]\\
			&\leq C^* \Big[ \norm{h}_{L^\infty(\R^{n-1})}  + S^\ast \norm{(-\Delta)^\frac{1}{4} h}_{L^{2(n-1)}(\R^{n-1})} \Big] \norm{u^-(0)}_{\dot{H}^{1/2}(\R^{n-1})}.
		\end{split}
	\end{equation*}
Using the previous estimate in~\eqref{eq:term_I},  we obtain
\begin{equation}\label{eq:term_I_improved}
	\begin{split}
	\abs{I}&\leq 2C^* \Big[ \norm{x'\, \alpha}_{L^\infty(\R^{n-1})}  + S^\ast \norm{D^\frac{1}{2} (x'\, \alpha)}_{L^{2(n-1)}(\R^{n-1})} \Big] \norm{u^-(0)}_{\dot{H}^{1/2}(\R^{n-1})}^2\\
	&\leq 2C^* \big[b_1 + S^\ast b_2 \big] \norm{u^-(0)}_{\dot{H}^{1/2}(\R^{n-1})}^2, 
	\end{split}
\end{equation} 	
where in the last inequality 
we have used the hypotheses~\eqref{eq:alpha_condition} 
and~\eqref{eq:fractional_derivative}.
Using additionally~\eqref{eq:homogeneous_trace} for~$u^-$ 
in~\eqref{eq:term_I_improved}, one gets
\begin{equation*}
	\abs{I}\leq 2C^*[b_1 + S^\ast b_2]\int_\Omega \abs{\nabla u^-}\, dx.
\end{equation*}
Using the last estimate in~\eqref{eq:crucial_identity_f0}, we get
	\begin{multline*}
		\hspace{0.8cm}\Big(1- 2 C^*\left[b_1 +S^\ast b_2 \right] \Big) \int_\Omega \abs{\nabla u^-}^2\, dx + \frac{n-3}{n-1}(\Re \lambda)^{-\frac{1}{2}}\abs{\Im \lambda} \int_{\Omega} \abs{x}\abs{\nabla u^-}^2\, dx\\
		+(n-1)\int_{\partial \Omega} \Re \alpha \abs{u}^2\, d\sigma +(\Re \lambda)^{-\frac{1}{2}} \abs{\Im \lambda} \int_{\partial \Omega} \abs{x} \Re \alpha \abs{u}^2\, d\sigma\leq 0.
	\end{multline*}
Since we are assuming~\eqref{hp:non_self-adjoint},
we have $\Re \alpha\geq 0$,
so discarding non-negative terms, we obtain
	\begin{equation*}
	\Big(1- 2 C^*\left[b_1 +S^\ast b_2 \right] \Big) \int_\Omega \abs{\nabla u^-}^2\, dx\leq 0.
	\end{equation*}
By virtue of~\eqref{eq:smallness}, 
it follows that~$u^-$ and so~$u$ are identically equal to zero.

\subsubsection*{$\bullet$ Case $\abs{\Im \lambda} > \Re \lambda.$}
Using Lemma~\ref{lemma:5identities}
for the approximating sequence~$u_R$,
let us now consider the sum
	\begin{equation*}
		\eqref{eq:first_Omega_const} \,\pm\, \eqref{eq:first_Omega_im_const},
	\end{equation*}
which gives
	\begin{equation*}
		\begin{split}
		\int_{\Omega} \abs{\nabla u_R}^2\, dx + \int_{\partial \Omega} \big(\Re \alpha \pm \Im \alpha\big) \abs{u_R}^2\, d\sigma=\ &(\Re \lambda \pm \Im \lambda) \int_{\Omega} \abs{u_R}^2\, dx\\
		& + \Re \int_{\Omega} \widetilde{f}_R \bar{u}_R\, dx \pm \Im \int_{\Omega} \widetilde{f}_R \bar{u}_R\, dx\\
		& + \Re \int_{\partial \Omega} \widetilde{g}_R \bar{u}_R\, d\sigma \pm \Im \int_{\partial \Omega} \widetilde{g}_R \bar{u}_R\, d\sigma.
		\end{split}
	\end{equation*}
	Estimating the terms involving $\widetilde{f}_R$ and $\widetilde{g}_R$ as~$I_1$ and~$I_3$ in the proof of Theorem~\ref{thm:self-adjoint} in Section~\ref{Sec.virial} 
and letting $R\to \infty,$ by the dominated convergence theorem, one gets
	\begin{equation*}
		\int_{\Omega} \abs{\nabla u}^2\, dx + \int_{\partial \Omega} \big(\Re \alpha \pm \Im \alpha\big) \abs{u}^2\, d\sigma=(\Re \lambda \pm \Im \lambda) \int_{\Omega} \abs{u}^2\, dx.
	\end{equation*}
	Using hypothesis~\eqref{hp:non_self-adjoint} one easily gets
	\begin{equation*}
		\int_{\Omega} \abs{\nabla u}^2\, dx\leq (\Re \lambda \pm \Im \lambda) \int_{\Omega} \abs{u}^2\, dx.
	\end{equation*}
	Therefore $\Re \lambda \pm \Im \lambda\geq 0$ unless $u=0.$ But since $\abs{\Im \lambda}> \Re \lambda$ we conclude that $u=0.$

\medskip
This concludes the proof of Theorem~\ref{thm:non_self-adjoint}.
\qed

\subsection{Non-self-adjoint setting: 
Proof of Theorem~\ref{thm:non_self-adjoint_h}}
As above, let $u \in \mathcal{D}(-\Delta_\alpha^\Omega)$ be any solution
to the eigenvalue equation $-\Delta_\alpha^\Omega u = \lambda u$. 
The proof is again split into two cases: $\abs{\Im \lambda}\leq \Re \lambda$ and $\abs{\Im \lambda}>\Re \lambda.$ Since the latter can be treated exactly as 
in the proof of Theorem~\ref{thm:non_self-adjoint} in Section~\ref{Sec.nsa1},
here we consider only the former.

\subsubsection*{$\bullet$ Case $\abs{\Im \lambda}\leq \Re \lambda.$}
As a starting point we consider inequality~\eqref{eq:crucial_identity_h} 
in Lemma~\ref{crucial_identity_h} with $f=0$, namely
	\begin{multline}\label{eq:starting_h}
	\qquad \qquad \int_{\Omega} \abs{\nabla u^-}^2\, dx + \frac{n-3}{n-1}(\Re\lambda)^{-{1/2}} \abs{\Im \lambda} \int_{\Omega} \abs{x} \abs{\nabla u^-}^2\, dx\\
	-\int_{\partial \Omega} x\cdot \nabla \Re \alpha \abs{u}^2\, d\sigma
	- 2 \Im \int_{\partial \Omega} (x\, \Im \alpha\, u^-)\cdot \overline{\nabla\, u^-}\, d\sigma
	+(\Re \lambda)^{-{1/2}} \abs{\Im \lambda} \int_{\partial \Omega} \abs{x} \Re \alpha \abs{u}^2\, d\sigma\leq 0.
	\end{multline}
In order to estimate the following term in~\eqref{eq:starting_h}
\begin{equation*}
	I:=-2\Im \int_{\partial \Omega} (x\, \Im\alpha\, u^-)\, \cdot \overline{\nabla\, u^-}\, d\sigma=-2\Im \int_{\R^{n-1}} (x'\, \Im\alpha\, u^-)\, \cdot \overline{\nabla_{x'}\, u^-}\, dx',
\end{equation*}
we introduce the sesquilinear form
	\begin{equation*}
		T(f,g):=\int_{\R^{n-1}} (x'\, \Im \alpha\, f)\cdot \nabla_{x'}\, \bar{g}\, d\sigma
	\end{equation*}
with suitable functions $f,g:\R^{n-1}\to\C$
and use an interpolation result. 	
Firstly, by virtue of the H\"older and Cauchy-Schwarz inequalities 
and by using hypothesis~\eqref{eq:alpha_condition_h}, we get
	\begin{equation}\label{eq:interpolation_1}
		\abs{T(f,g)}\leq \norm{x' \Im \alpha}_{L^\infty(\R^{n-1})} \norm{f}_{L^2(\R^{n-1})} \norm{g}_{\dot{H}^1(\R^{n-1})}\leq b_1 \norm{f}_{L^2(\R^{n-1})} \norm{g}_{\dot{H}^1(\R^{n-1})}.
	\end{equation}
Second, making instead use of an integration by parts, we also have
	\begin{equation*}
		T(f,g)=-\int_{\partial \Omega} (x' \Im \alpha\, \bar{g}) \cdot \nabla f\, d\sigma-\int_{\partial \Omega} \div(x' \Im \alpha) f \bar{g}\, d\sigma,
	\end{equation*}
and consequently employing the H\"older and Cauchy-Schwarz inequalities
with help of~\eqref{eq:alpha_condition_h} and~\eqref{eq:Hardy_h}, one has
	\begin{equation}\label{eq:interpolation_2}
		\begin{split}
		\abs{T(f,g)}&\leq \norm{x' \Im \alpha}_{L^\infty(\R^{n-1})} \norm{f}_{\dot{H}^1(\R^{n-1})} \norm{g}_{L^2(\R^{n-1})} + \left(\int_{\partial \Omega} \abs{\div(x' \Im \alpha)}^2 \abs{f}^2 \, d\sigma \right)^\frac{1}{2}\norm{g}_{L^2(\R^{n-1})}\\
		&\leq (b_1 + b_2)\norm{f}_{\dot{H}^1(\R^{n-1})} \norm{g}_{L^2(\R^{n-1})}. 
		\end{split}
	\end{equation}
An interpolation between~\eqref{eq:interpolation_1} and~\eqref{eq:interpolation_2} 
gives
	\begin{equation*}
		\abs{T(f,g)}\leq [b_1(b_1 + b_2)]^\frac{1}{2} \norm{f}_{\dot{H}^\frac{1}{2}(\R^{n-1})} \norm{g}_{\dot{H}^\frac{1}{2}(\R^{n-1})}.
	\end{equation*}
	Hence, it follows that
	\begin{equation}\label{eq:I_2_h}
		\begin{split}
		\abs{I}&\leq 2 \abs{T(u^-(0),u^-(0))}\leq 2 [b_1(b_1 + b_2)]^\frac{1}{2} \norm{u^-(0)}_{\dot{H}^\frac{1}{2}(\R^{n-1})}^2\\
		&\leq 2 [b_1(b_1 + b_2)]^\frac{1}{2} \norm{\nabla u^-}_{L^2(\Omega)}^2,
		\end{split}
	\end{equation}
where in the last inequality we have used 
the Caffarelli-Sylvester extension~\eqref{eq:homogeneous_trace}.

Now, plugging~\eqref{eq:I_2_h} into~\eqref{eq:starting_h}, we have
		\begin{multline*}
		\hspace{0.8cm}\Big(1- 2 [b_1(b_1 + b_2)]^\frac{1}{2} \Big) \int_\Omega \abs{\nabla u^-}^2\, dx + \frac{n-3}{n-1}(\Re \lambda)^{-\frac{1}{2}}\abs{\Im \lambda} \int_{\Omega} \abs{x}\abs{\nabla u^-}^2\, dx\\
		-\int_{\partial \Omega} x \cdot \nabla_{x}\Re \alpha\, \abs{u}^2\, d\sigma +(\Re \lambda)^{-\frac{1}{2}} \abs{\Im \lambda} \int_{\partial \Omega} \abs{x} \Re \alpha \abs{u}^2\, d\sigma\leq 0.
	\end{multline*}
Since we are assuming~\eqref{hp:non_self-adjoint}, we have $\Re \alpha\geq 0$.
Using additionally~\eqref{hp:repulsivity_h} 
and discarding non-negative terms, we obtain
	\begin{equation*}
	\Big(1- 2 [b_1(b_1 + b_2)]^\frac{1}{2} \Big) \int_\Omega \abs{\nabla u^-}^2\, dx\leq 0.
	\end{equation*}
By virtue of~\eqref{eq:smallness_h}, it follows that $u^-$ and so $u$ are identically equal to zero and this conclude our proof.
\qed 

\medskip
Now we turn to the resolvent estimates
contained in Theorems~\ref{thm:main_resolvent} and~\ref{thm:resolvent}.


\subsection{Resolvent estimates: Auxiliary result}
The proof of Theorem~\ref{thm:resolvent} (in fact also of Theorem~\ref{thm:non_self-adjoint}) in the case $\lambda=0$,
or more generally if $\Im \lambda=0,$ 
can be proved easily proceeding as in the self-adjoint framework, in other words there is no need to introduce the auxiliary function $u^-.$ 
Therefore, from now on, 
we will assume without loss of generality $\abs{\Im \lambda}>0.$

We shall need the following \emph{a priori} $L^2$-bound 
for weak solutions to~\eqref{eq:resolvent}.  

\begin{lemma}
Let $n\geq 1$ and assume that $u\in \widetilde{\mathcal{D}}$ is a weak solution to~\eqref{eq:resolvent}. Then the estimate
	\begin{equation}\label{L^2_bound}
			\norm{u}_{L^2(\Omega)}^2\leq \abs{\Im\lambda}^{-1} \left( \int_{\partial \Omega} \abs{\Im \alpha} \abs{u}^2\, d\sigma  + \int_{\Omega} \abs{f} \abs{u}\, dx \right)	
	\end{equation}
	holds true.
\end{lemma}
\begin{proof}
Since we want to make use of Lemma~\ref{lemma:5identities},
we again consider the compactly supported approximation~$u_R$ 
introduced in Section~\ref{Sec.approx}. 
Recall that $u_R$ solves~\eqref{eq:resolvent_approx}.	
Identity~\eqref{eq:first_Omega_im_const} reads
\begin{equation*}
			\int_{\partial \Omega} \Im \alpha \abs{u_R}^2\, d\sigma= \Im \lambda \int_{\Omega} \abs{u_R}^2\, dx
			+ \Im \int_{\Omega} \widetilde{f}_R \bar{u}_R\, dx + \Im \int_{\partial \Omega} \widetilde{g}_R \bar{u}_R\, d\sigma.
\end{equation*}
Estimating the right-hand side by making use of the Cauchy-Schwarz inequality 
and Lemma~\ref{lemma:fg} and passing to the limit $R\to \infty,$ 
by means of the dominated convergence theorem, one gets
		\begin{equation*}
			\int_{\partial \Omega} \Im \alpha \abs{u}^2\, d\sigma= \Im \lambda \int_{\Omega} \abs{u}^2\, dx
			+ \Im \int_{\Omega} f \bar{u}\, dx.
		\end{equation*}
		Multiplying the previous by $\sgn(\Im \lambda)$ we obtain~\eqref{L^2_bound}.
\end{proof}

\subsection{Resolvent estimates: Proof of Theorem~\ref{thm:resolvent}}
Let $u \in \mathcal{D}(-\Delta_\alpha^\Omega)$ be any solution
to the resolvent equation $(-\Delta_\alpha^\Omega-\lambda)u=f$.   
We split the proof into two cases: $\abs{\Im\lambda}\leq \Re \lambda$ and $\abs{\Im \lambda}> \Re \lambda.$
In order to save space we write $\norm{\cdot}$ instead of $\norm{\cdot}_{L^2(\Omega)}$;
the norms in other spaces will be written explicitly.

\subsubsection*{$\bullet$ Case $\abs{\Im\lambda}\leq \Re \lambda$.}
We start by estimating the individual terms on the right-hand side
of the key identity~\eqref{eq:crucial_identity}, namely
$$
  F_1 := (n-1) \Re \int_{\Omega} f \bar{u}\, dx, \qquad 
  F_2 := 2 \Re \int_{\Omega} x f^-\cdot \overline{\nabla u^-}\, dx, \qquad 
  F_3 := (\Re \lambda)^{-{1/2}} \abs{\Im \lambda} \Re \int_{\Omega} f \abs{x} \bar{u}\, dx. 
$$
By the Cauchy-Schwarz inequality, 
the Hardy inequality~\eqref{eq:Hardy} 
and by using that $\abs{u}=\abs{u^-}$, 
we have
\begin{equation*}
			\abs{F_1}\leq (n-1)\norm{\abs{x} f} \norm*{\frac{u}{\abs{x}}}\leq \frac{2(n-1)}{n-2} \norm{\abs{x} f}\norm{\nabla u^-}.
\end{equation*}
At the same time,
\begin{equation*}
			\abs{F_2}\leq 2\norm{\abs{x} f}\norm{\nabla u^-},
\end{equation*}
where we have used that $\abs{f^-}= \abs{f}.$
Finally, using~\eqref{L^2_bound}, 
the fact that $\abs{\Im \lambda}\leq \Re \lambda,$ the Cauchy-Schwarz inequality,
the Hardy inequality~\eqref{eq:Hardy} and $\abs{u}=\abs{u^-}$, we obtain
\begin{equation*}
			\abs{F_3}\leq \norm{\abs{x} f} \left( \int_{\partial \Omega} \abs{\Im \alpha} \abs{u}^2\, d\sigma \right)^{1/2} + \frac{\sqrt{2}}{\sqrt{n-2}} \norm{\abs{x}f}^{3/2}\norm{\nabla u^-}^{1/2}.
\end{equation*}
Summing up, the right-hand side of~\eqref{eq:crucial_identity}
can be estimated as follows:
\begin{multline*}
			\hspace{1cm}F_1+F_2+F_3\\
			\leq \frac{2(2n-3)}{n-2}\norm{\abs{x} f} \norm{\nabla u^-} + \frac{\sqrt{2}}{\sqrt{n-2}} \norm{\abs{x}f}^{3/2}\norm{\nabla u^-}^{1/2}
			+ \norm{\abs{x} f} \left( \int_{\partial \Omega} \abs{\Im \alpha} \abs{u}^2\, d\sigma \right)^{1/2}.
\end{multline*}
Given $\varepsilon, \delta>0$ and making use of the Young inequality, 
one has
\begin{equation*}
\norm{\abs{x}f}\norm{\nabla u^-}\leq \frac{1}{2\varepsilon^2} \norm{\abs{x}f}^2 + \frac{\varepsilon^2}{2}\norm{\nabla u^-}^2,
\qquad 
\norm{\abs{x}f}^\frac{3}{2} \norm{\nabla u^-}^\frac{1}{2}\leq \frac{3}{4\delta^\frac{4}{3}} \norm{\abs{x}f}^2 + \frac{\delta^4}{4}\norm{\nabla u^-}^2
\end{equation*}
and
\begin{equation*}
			\norm{\abs{x} f} \left( \int_{\partial \Omega} \abs{\Im \alpha} \abs{u}^2\, d\sigma \right)^{1/2}\leq \frac{1}{2} \norm{\abs{x}f}^2 + \frac{1}{2}\int_{\partial \Omega} \abs{\Im \alpha}\abs{u}^2\, d\sigma.
\end{equation*}
Consequently,
\begin{equation}\label{eq:I_1,I_2,I_3}
			\begin{split}
			\hspace{1cm} F_1+F_2+F_3\leq & \left(\frac{1}{\varepsilon^2} \frac{2n-3}{n-2} + \frac{3}{4\delta^\frac{4}{3}} \frac{\sqrt{2}}{\sqrt{n-2}} + \frac{1}{2} \right) \norm{\abs{x}f}^2\\
			&+\left( \varepsilon^2 \frac{2n-3}{n-2} + \frac{\delta^4}{4} \frac{\sqrt{2}}{\sqrt{n-2}} \right) \norm{\nabla u^-}^2 +\frac{1}{2} \int_{\partial \Omega} \abs{\Im \alpha}\abs{u}^2\, d\sigma.
			\end{split}
\end{equation}
As in the proof of Theorem~\ref{thm:non_self-adjoint} in Section~\ref{Sec.nsa1},
we have
		\begin{equation*}
			- \frac{(n-1)}{2}(\Re \lambda)^{-\frac{1}{2}}  \abs{\Im \lambda} \int_{\Omega} \frac{~\abs{u}^2}{\abs{x}}\, dx\geq -\frac{2}{n-1}(\Re \lambda)^{-{1/2}} \abs{\Im \lambda} \int_{\Omega} \abs{x}\abs{\nabla u^-}^2\, dx
		\end{equation*}
		and 
		\begin{equation*}
			2 \Re \int_{\partial \Omega}(x\, \alpha\, u^-)\cdot \overline{\nabla u^-}\, d\sigma\geq -2C^*[b_1 + S^\ast b_2] \norm{\nabla u^-}^2.
		\end{equation*}
Using these estimates and~\eqref{eq:I_1,I_2,I_3} in~\eqref{eq:crucial_identity}, 
we get  
\begin{multline*}
				\left( 1- \varepsilon^2 \frac{2n-3}{n-2} - \frac{\delta^4}{4} \frac{\sqrt{2}}{\sqrt{n-2}}-2C^*[b_1 + S^\ast b_2] \right)\norm{\nabla u^-}^2 + (\Re \lambda)^{-{1/2}} \abs{\Im \lambda} \frac{n-3}{n-1} \int_{\Omega} \abs{x} \abs{\nabla u^-}^2\, dx \\
				+\int_{\partial \Omega} \big[(n-1) \Re \alpha - \frac{1}{2}\abs{\Im \alpha} \big] \abs{u}^2\, d\sigma + (\Re \lambda)^{- {1/2}} \abs{\Im \lambda} \int_{\partial \Omega} \abs{x} \Re \alpha \abs{u}^2\, d\sigma\\
				\leq \left(\frac{1}{\varepsilon^2} \frac{2n-3}{n-2} + \frac{3}{4\delta^\frac{4}{3}} \frac{\sqrt{2}}{\sqrt{n-2}} + \frac{1}{2} \right) \norm{\abs{x}f}^2.	
\end{multline*}
Since we are assuming~\eqref{hp:non_self-adjoint}, 
we particularly have $(n-1)\Re \alpha - \frac{1}{2}\abs{\Im \alpha}\geq 0$ 
and $\Re\alpha\geq 0$. Therefore discarding non-negative terms, we have
\begin{equation*}
			\left( 1- \varepsilon^2 \frac{2n-3}{n-2} - \frac{\delta^4}{4} \frac{\sqrt{2}}{\sqrt{n-2}}-2C^*[b_1 + S^\ast b_2] \right)\norm{\nabla u^-}^2\leq \left(\frac{1}{\varepsilon^2} \frac{2n-3}{n-2} + \frac{3}{4\delta^\frac{4}{3}} \frac{\sqrt{2}}{\sqrt{n-2}} + \frac{1}{2} \right) \norm{\abs{x}f}^2.
\end{equation*}
Choosing $\varepsilon, \delta$ small enough 
and using that~\eqref{eq:smallness} holds,
estimate~\eqref{eq:res_1} is proved.

\subsubsection*{$\bullet$ Case $\abs{\Im \lambda}> \Re \lambda$.}
Identity~\eqref{eq:first_Omega_const} for~$u_R$ reads
\begin{equation*}
			\int_{\Omega} \abs{\nabla u_R}^2\, dx + \int_{\partial \Omega} \Re \alpha \abs{u_R}^2\, d\sigma = \Re \lambda \int_{\Omega} \abs{u_R}^2\, dx + \Re \int_{\Omega} \widetilde{f}_R \bar{u}_R\, dx + \Re \int_{\partial \Omega} \widetilde{g}_R \bar{u}_R\, d\sigma.
\end{equation*}
After passing to the limit $R\to \infty$ making use of the estimates in Lemma~\ref{lemma:fg}, one gets 
		\begin{equation*}
			\int_{\Omega} \abs{\nabla u}^2\, dx + \int_{\partial \Omega} \Re \alpha \abs{u}^2\, d\sigma = \Re \lambda \int_{\Omega} \abs{u}^2\, dx + \Re \int_{\Omega} f \bar{u}\, dx.
		\end{equation*}
		If $\Im \lambda=0$ then in particular $\Re \lambda <0$ and estimate~\eqref{eq:res_2} follows easily from the previous identity. 
		If $\Im \lambda\neq 0,$ we use the $L^2$-control~\eqref{L^2_bound}. 
Consequently, since $\abs{\Im \lambda}>\Re \lambda,$ after using the Cauchy-Schwarz and the Hardy inequality~\eqref{eq:Hardy}, we have
\begin{equation*}
			\int_{\Omega} \abs{\nabla u}^2\, dx + \int_{\partial \Omega} (\Re \alpha - \abs{\Im \alpha}) \abs{u}^2\, d\sigma\leq \frac{4}{n-2} \norm{\abs{x} f}\norm{\nabla u}.
\end{equation*}
By hypothesis~\eqref{hp:non_self-adjoint},
the desired estimate~\eqref{eq:res_2} follows. 
\qed

\subsection{Resolvent estimates: Proof of Theorem\ref{thm:main_resolvent}}
Now we are in a position to prove Theorem\ref{thm:main_resolvent}
as a corollary of Theorem~\ref{thm:resolvent}.
If $\abs{\Im \lambda}\leq \Re \lambda$,
then 
\begin{equation*}
\norm{\abs{x}^{-1} u}\leq \frac{2}{n-2}\norm{\nabla u^-}
\leq \frac{2c}{n-2} \norm{\abs{x}f},
\end{equation*}
where we have used~\eqref{eq:res_1}, the fact $\abs{u^-}=\abs{u}$ 
and the Hardy inequality~\eqref{eq:Hardy}. 
If $\abs{\Im \lambda}> \Re \lambda,$ 
using~\eqref{eq:res_2} and again the Hardy inequality~\eqref{eq:Hardy}, 
we also arrive at
			\begin{equation*}
				\norm{\abs{x}^{-1} u}\leq \frac{2c}{n-2}\norm{\abs{x}f}.
			\end{equation*}
This concludes the proof.
\qed

		
\appendix


\section{The Robin Laplacian}\label{Appendix:rigorous_definition}
For the convenience of the reader, here we provide details 
on the rigorous definition of the Robin Laplacian in the half-space~$\Omega$ 
as an m-sectorial operator in $L^2(\Omega)$
and characterise its operator domain.
 	
Given $\alpha \in L^\infty(\partial \Omega; \C)$,
let us consider the quadratic form~$h_\alpha$ 
introduced in~\eqref{form}.
We consider~$h_\alpha$ as a perturbation of~$h_0$,
\emph{i.e.}, we write $h_\alpha = h_0 + \mathsf{a}$ with
\begin{equation*}
\mathsf{a}[u]:=\int_{\partial \Omega} \alpha \abs{u}^2\, d\sigma, 
\qquad \mathcal{D}(\mathsf{a}):= H^1(\Omega).
\end{equation*}	
Observe that $h_0$ is a densely defined, non-negative 
and closed form in $L^2(\Omega)$.
In fact, $h_0$~is associated 
with the (self-adjoint) Neumann Laplacian in $L^2(\Omega)$.
We claim that the form~$\mathsf{a}$ is relatively bounded 
with respect to~$h_0$ with the relative bound less than one, \emph{i.e.},
$\mathcal{D}(h_0)\subseteq\mathcal{D}(\mathsf{a}) $ 
(which is trivially true in our case) 
and there exist two real constants $a<1$ and~$b$ such that
\begin{equation}\label{eq:relatively_bounded}
	\abs*{\mathsf{a}[u]}\leq a\, h_0[u] + b\norm{u}_{L^2(\Omega)} \quad \text{for all}\, u\in \mathcal{D}(h_0).
\end{equation} 
To check this inequality, we shall prove a preliminary estimate for the $L^2$-norm of $u$ on $\partial \Omega.$
First, let us start considering $u\in C^\infty_0(\R^n),$ it is easy to show that
\begin{equation}\label{eq:r.bound-boundary}
	\begin{split}
		\abs{u(x', 0)}^2
		&=- \int_0^\infty \frac{\partial}{\partial_{x_n}} \abs{u(x', x_n)}^2\Big|_{x_n=t}\, dt\\
		&= -2 \Re \int_0^\infty \bar{u}_{x_n}(x',t) u(x',t)\, dt\\
		&\leq \int_0^\infty \Big( \varepsilon \abs{u_{x_n}}^2(x',t) + \frac{1}{\varepsilon} \abs{u}^2(x',t) \Big)\, dt,			
	\end{split}
\end{equation}
where in the last inequality we have used that for any $a,b\geq0$ and for any $\varepsilon>0,$ one has $ab=\sqrt{\varepsilon} a \frac{1}{\sqrt{\varepsilon}} b \leq \varepsilon\frac{a^2}{2} + \frac{1}{\varepsilon}\frac{b^2}{2}.$ Integrating~\eqref{eq:r.bound-boundary} over $\R^{n-1},$ one gets
\begin{equation}\label{eq:useful-bound}
	\norm{u}_{L^2(\partial \Omega)}^2 \leq \varepsilon \norm{\nabla u}_{L^2(\Omega)}^2 + \frac{1}{\varepsilon} \norm{u}_{L^2(\Omega)}^2
	\quad \text{for all}\, u\in C^\infty_0(\R^n).
\end{equation} 
Since $C^\infty_0(\R^n)$ are dense in $H^1(\Omega),$ inequality~\eqref{eq:useful-bound} holds true for any $u\in H^1(\Omega)=\mathcal{D}(h_0).$ Now, with inequality~\eqref{eq:useful-bound}
at hands, estimate~\eqref{eq:relatively_bounded} follows immediately. Indeed,
\begin{equation*}
			\begin{split}
			\abs{\mathsf{a}[u]}&\leq \int_{\partial \Omega} \abs{\alpha} \abs{u}^2\, d\sigma\\
												&\leq \norm{\alpha}_{L^\infty(\partial \Omega)} \norm{u}_{L^2(\partial\Omega)}^2\\
												&\leq \varepsilon \norm{\alpha}_{L^\infty(\partial \Omega)} \norm{\nabla u}_{L^2(\Omega)}^2 + \frac{1}{\varepsilon} \norm{a}_{L^\infty(\partial \Omega)} \norm{u}_{L^2(\Omega)}^2.
			\end{split}
\end{equation*}
Choosing $\varepsilon>0$ such that $\varepsilon\norm{\alpha}_{L^\infty(\partial \Omega)}<1$, then~\eqref{eq:relatively_bounded} follows.	

As a consequence of the validity of~\eqref{eq:relatively_bounded},
the sum $h_\alpha=h_0+\mathsf{a}$ is densely defined, closed and sectorial
due to the stability result \cite[Thm.~VI.1.33]{Kato}. 
By the representation theorem \cite[Thm.~VI.2.1]{Kato}, 
there exists an m-sectorial operator $-\Delta_\alpha^\Omega$ in $L^2(\Omega)$
such that 
		\begin{equation*}
			\mathcal{D}(-\Delta_\alpha^\Omega)=\big\{ u\in \mathcal{D}(h_\alpha) \big | \, \text{there exists a}\; w_u\in L^2(\Omega) \; \text{such that}\; h_\alpha[u,v]=\langle w_u, v\rangle_{L^2(\Omega)}\; 
\text{for all}\; v\in \mathcal{D}(h_\alpha)\big \}
		\end{equation*}
		and
		\begin{equation*}
			-\Delta_\alpha^\Omega u= w_u, \quad u\in \mathcal{D}(-\Delta_\alpha^\Omega)
		\end{equation*}
Here $h_\alpha[\cdot,\cdot]$ is the sesquilinear form 
associated with the quadratic form~$h_\alpha$
and we use the convention that inner products 
are conjugate linear in the second argument.
The operator $-\Delta_\alpha^\Omega$ is called the Robin Laplacian 
(the terminology is justified by Theorem~\ref{Thm.Robin} below).
 
Now we turn to the description of the domain 
of the m-sectorial operator $-\Delta_\alpha^\Omega.$
In order to do that, let us recall some preliminary basic facts on traces of functions from Sobolev spaces (if~$\Omega$ is an open, bounded, non-empty and Lipschitz domain, these results can be found in~\cite{Gesztesy-Mitrea_2008}, in the case of the half-space the proof follows a similar strategy).
Let $\eta$ be the outward pointing normal unit vector to $\partial \Omega,$ the Dirichlet trace map $\gamma_D^0\colon C^\infty(\overline{\Omega})\to C^\infty(\partial \Omega),$ $\gamma_D^0u=u|_{\partial \Omega}$ and the Neumann trace map $\gamma_N^0\colon C^\infty(\overline{\Omega})\to C^\infty(\partial \Omega),$ $\gamma_N^0u=\eta \cdot \nabla u|_{\partial \Omega}=-u_{x_n}|_{\partial \Omega},$  extend to
		\begin{equation*}
			\gamma_D \colon H^s(\Omega)\to H^{s-1/2}(\partial \Omega)\hookrightarrow L^2(\partial\Omega),\quad 1/2<s<3/2
		\end{equation*}
		and
		\begin{equation*}			\gamma_N= \eta \cdot \gamma_D \nabla \colon H^{s+1}(\Omega) \to L^2(\partial \Omega), \quad 1/2<s<3/2,
		\end{equation*}
		respectively.
		The constraint $1/2<s<3/2$ is usually too restrictive for applications, hence one is interested in extending the action of $\gamma_D$ and $\gamma_N$ to other settings. Moving in this direction it has been proved that for any $s>-3/2$ the restriction-to-boundary operator $\gamma_D^0$ extends to a linear bounded operator
			\begin{equation*}
				\gamma_D\colon \big\{ u\in H^{1/2}(\Omega)\, \big | \, \Delta u\in H^s(\Omega)\big\} \to L^2(\partial \Omega),
			\end{equation*}
			when $\big\{ u\in H^{1/2}(\Omega)\, \big | \, \Delta u\in H^s(\Omega)\big\}$ is equipped with the natural graph norm $u\mapsto \norm{u}_{H^{1/2}(\Omega)} + \norm{\Delta u}_{H^s(\Omega)}.$
With this result at hand, the next lemma easily follows. 

\begin{lemma}\label{lemma:Neumann_3/2epsilon}
			The Neumann trace operator $\gamma_N^0$ also extends to
			\begin{equation*}
				\gamma_N= \eta \cdot \gamma_D \nabla \colon \big\{ u\in H^{3/2}(\Omega)\, \big|\, \Delta u \in L^2(\Omega) \big\} \to L^2(\partial \Omega)
			\end{equation*}
			in a bounded fashion when the space $\big\{ u\in H^{3/2}(\Omega)\, \big|\, \Delta u \in L^2(\Omega) \big\}$ is equipped with the natural graph norm $u\mapsto \norm{u}_{H^{3/2}(\Omega)} + \norm{\Delta u}_{L^2(\Omega)}.$ Moreover, given $u\in H^{3/2}(\Omega)$ with $\Delta u \in L^2(\Omega),$ the generalised Green identity
			\begin{equation}\label{eq:Green_identity}
				\int_{\partial \Omega} \eta \cdot \gamma_D(\nabla u) \bar{v}\, d\sigma=\int_{\Omega} \nabla u \overline{\nabla v}\, dx + \int_{\Omega} \Delta u \bar{v}\, dx 
			\end{equation}
			holds true for all $v\in H^1(\Omega).$ 
\end{lemma}

Now we are in position to prove the following result,
which in particular states that the Robin Laplacian 
indeed acts as the (weak) Laplacian and satisfies the Robin
boundary conditions (in the sense of traces).
\begin{theorem}\label{Thm.Robin}
If $\alpha \in L^\infty(\partial \Omega; \C),$ 
then the Robin Laplacian $-\Delta_\alpha^\Omega$ satisfies
			\begin{equation*}
				-\Delta_\alpha^\Omega u=-\Delta u \qquad \text{for all}\; u \in \mathcal{D}(-\Delta_\alpha^\Omega),
			\end{equation*}
			where 
			\begin{equation*}
				\mathcal{D}(-\Delta_\alpha^\Omega)=\big\{u\in H^{3/2}(\Omega)\big|\, \Delta u \in L^2(\Omega),\, -u_{x_n} +\alpha u=0\; \text{in}\; L^2(\partial \Omega) \big\}.
			\end{equation*}
\end{theorem}

\begin{proof}
Even though the proof of this result follows a standard scheme 
(see, \emph{e.g.}, \cite{Gesztesy-Mitrea_2008} or \cite{Helffer-Pankrashkin_2015}), 
we will provide it for the sake of completeness 
(see also~\cite[Rem.~7.5~ii)]{Behrndt-Langer-Lotoreichik-Rohleder_2018}  
and~\cite[Thm.~3.5]{Behrndt-Schlosser} for the proof of 
a slightly more general result in the case of bounded domains
performed using the technique of boundary triples).

Recall the definition of the space~$\widetilde{\mathcal{D}}$ given in~\eqref{space}.
As a starting point we prove that 
$\mathcal{D}(-\Delta_\alpha^\Omega)\subseteq \widetilde{\mathcal{D}}.$
Let $z\in \C\setminus \sigma(-\Delta_\alpha^\Omega)$ and consider the following problem: given $w\in L^2(\Omega),$ find $u\in \mathcal{D}(h_\alpha)$ such that
\begin{equation}\label{eq:domain_z}
  h_\alpha[u,v] -z\,\langle u,v \rangle_{L^2(\Omega)}
  =\langle w,v \rangle_{L^2(\Omega)} 
  \quad \text{for all}\; v\in \mathcal{D}(h_\alpha).
\end{equation}
By definition, equation~\eqref{eq:domain_z} holds if and only if $u\in \mathcal{D}(-\Delta_\alpha^\Omega)$ and $-\Delta_\alpha^\Omega u=z u + w.$
Set
\begin{equation*}
	u:=(-\Delta_\alpha^{\Omega}-z)^{-1}w,
\end{equation*} 
the unique solution to~\eqref{eq:domain_z} (whose existence follows from elementary Hilbert theory),
then the boundedness of the resolvent operator
($z\in \C\setminus\sigma(-\Delta_\alpha^\Omega)$)
\begin{equation*}
	(-\Delta_\alpha^\Omega -z)^{-1}\colon L^2(\Omega) \to \widetilde{\mathcal{D}}
\end{equation*}  
(which follows by adapting  
the argument provided in~\cite{Gesztesy-Mitrea_2008} to the half-space)
gives the claim.

\smallskip
Now observe that for any $u\in \mathcal{D}(-\Delta_\alpha^\Omega),$ then 
\begin{equation}\label{equality_distribution}
	-\Delta_\alpha^\Omega u=-\Delta u\quad \text{in}\quad \mathscr{D}'(\Omega).
\end{equation}
Indeed, if $u\in \mathcal{D}(-\Delta_\alpha^\Omega),$ then there exists $w_u\in L^2(\Omega)$ such that 
\begin{equation*}
	\int_\Omega \nabla u \cdot \overline{\nabla v}\, dx + \int_{\partial \Omega} \alpha u \bar{v}\, d\sigma=\int_{\Omega} w_u \bar{v}\, dx=:\int_{\Omega} -\Delta_\alpha^\Omega u \bar{v}\, dx \quad \text{for all}\quad v\in H^1(\Omega).
\end{equation*}
On the other hand, taking $v\in C^\infty_0(\Omega)\hookrightarrow H^1(\Omega),$ one has
\begin{equation*}
	\int_{\Omega} \nabla u \cdot \overline{\nabla v}\, dx + \int_{\partial \Omega} \alpha u \bar{v}\, d\sigma= -\int_{\Omega} \Delta u \bar{v}\, dx\quad \text{for all}\quad v\in C^\infty_0(\Omega)
\end{equation*}
and hence~\eqref{equality_distribution} follows. 

Going further, suppose that $u\in \mathcal{D}(-\Delta_\alpha^\Omega),$ since we have proved that then $u\in \widetilde{\mathcal{D}},$ hence~\eqref{eq:Green_identity} holds. Using also~\eqref{equality_distribution} and compute
\begin{equation*}
	\begin{split}
		\int_{\Omega} \nabla u \cdot \overline{\nabla v}\, dx&=-\int_{\Omega} \Delta u \bar{v}\, dx - \int_{\partial \Omega} u_{x_n}\bar{v}\, d\sigma\\
		&=\int_\Omega -\Delta_\alpha^\Omega u \bar{v}\, dx -\int_{\partial \Omega} \alpha u \bar{v}\, d\sigma + \int_{\partial \Omega} (-u_{x_n} + \alpha u)\bar{v}\, d\sigma\\
		&=\int_{\Omega} \nabla u \cdot \overline{\nabla v}\, dx + \int_{\partial \Omega} (-u_{x_n} + \alpha u)\bar{v}\, d\sigma,
	\end{split}
\end{equation*} 
we get
\begin{equation*}
	\int_{\partial \Omega} (-u_{x_n} + \alpha u)\bar{v}\, d\sigma=0.
\end{equation*}
Since $v\in H^1(\Omega)$ is arbitrary, the map $\gamma_D\colon H^1(\Omega)\to H^{1/2}(\partial \Omega)$ is onto and via the density of $H^{1/2}(\partial \Omega)$ in $L^2(\partial \Omega),$ we conclude that $-u_{x_n} + \alpha u=0$ in $L^2(\partial \Omega).$ 

Thus, we have proved that
\begin{equation*}
	\mathcal{D}(-\Delta_\alpha^\Omega)\subseteq \big\{u\in H^{3/2}(\Omega)\ \big|\ \Delta u \in L^2(\Omega),\, -u_{x_n} +\alpha u=0\; \text{in}\; L^2(\partial \Omega) \big\}.
\end{equation*}

Next, assume $u\in \widetilde{\mathcal{D}}$ such that $ -u_{x_n} +\alpha u=0.$ Then from~\eqref{eq:Green_identity} one has
\begin{equation*}
	\begin{split}
	\int_{\Omega} \nabla u \cdot \overline{\nabla v}\, dx&=-\int_{\Omega} \Delta u \bar{v}\, dx -\int_{\partial \Omega} u_{x_n} \bar{v}\, d\sigma\\
	 &=-\int_{\Omega} \Delta u \bar{v}\, dx - \int_{\partial \Omega} \alpha u \bar{v}\, d\sigma.
	\end{split}
\end{equation*}
Hence $u\in \mathcal{D}(-\Delta_\alpha^\Omega)$ with $w_u:=-\Delta u$ and thus
\begin{equation*}
	\mathcal{D}(-\Delta_\alpha^\Omega) \supseteq \big\{u\in H^{3/2}(\Omega)\, \big|\, \Delta u\in L^2(\Omega),\, -u_{x_n} +\alpha u=0\; \text{in}\; L^2(\partial \Omega)\big\}. 
\end{equation*}
This concludes the proof of the theorem.			
\end{proof}


	\section{Method of multipliers: Proof of Lemma~\ref{lemma:5identities}}\label{Appendix:5_identities}
This part is concerned with the rigorous derivation of the identities
\eqref{eq:first_Omega_const}--\eqref{eq:second_Omega} in Lemma~\ref{lemma:5identities}. 
Let $u \in \widetilde{\mathcal{D}}_0$ solves~\eqref{eq:resolvent_gen}
and recall that $u \in H^1(\Omega)$ can be extended to a function 
(denoted by the same symbol) in $H^1(\R^n)$. 
The boundary-value problem~\eqref{eq:resolvent_gen} 
means that $u\in H^{3/2}(\R^n)$ such that $\Delta u\in L^2(\R^n)$ 
satisfies~\eqref{eq:weak_resolvent_gen} for any $v\in H^1(\R^n)$.
 
Choose in~\eqref{eq:weak_resolvent_gen} $v:=\varphi u,$ 
with $\varphi\colon \R^n \to \R,$ $\varphi\in W^{2,\infty}_\textup{loc}(\R^n)$  being a \emph{radial} function. Notice that since the support of~$u$ is compact, the hypotheses $\varphi, \nabla \varphi \in L^\infty_\textup{loc}(\R^n)$ guarantee that $v\in H^1(\R^n)$ and therefore $v$ is admissible as a test function in~\eqref{eq:weak_resolvent_gen}.
Using the Leibniz rule for weak derivatives in the first term of~\eqref{eq:weak_resolvent_gen} we get
		\begin{equation*}
			\int_{\Omega}  \varphi \abs{\nabla u}^2\, dx + \int_{\Omega} \bar{u} \nabla u \cdot \nabla \varphi\, dx + \int_{\partial \Omega} \alpha \varphi \abs{u}^2\, d\sigma
			= \lambda \int_{\Omega} \varphi \abs{u}^2\, dx + \int_{\Omega} {f} \varphi \bar{u}\, dx + \int_{\partial \Omega} {g} \varphi \bar{u}\, d\sigma. 
		\end{equation*}
		Taking the real part of the obtained identity gives 
		\begin{equation*}
		 \int_{\Omega} \varphi \abs{\nabla u}^2 \,dx
		 +\int_{\Omega} \Re(\bar{u}\nabla u)\cdot \nabla \varphi\, dx 
		 +  \int_{\partial \Omega} \Re \alpha \varphi \abs{u}^2\, d\sigma
		= \Re \lambda \int_{\Omega} \varphi \abs{u}^2\, dx + \Re \int_{\Omega} f \varphi \bar{u}\, dx + \Re \int_{\partial \Omega} g \varphi \bar{u}\, d\sigma.
	\end{equation*}	
	Consider the second term in the latter: using that 
	$
		\Re (\bar{u}\nabla u)=\frac{1}{2} \nabla \abs{u}^2
	$	
	and performing then an integration by parts, one gets
		\begin{multline*}
		-\frac{1}{2} \int_{\Omega} \Delta \varphi \abs{u}^2 \,dx + \int_{\Omega} \varphi \abs{\nabla u}^2 \,dx + \frac{1}{2} \int_{\partial \Omega} \abs{u}^2 \nabla \varphi \cdot \eta \,d\sigma +  \int_{\partial \Omega} \Re \alpha \varphi \abs{u}^2\, d\sigma\\
		= \Re \lambda \int_{\Omega} \varphi \abs{u}^2\, dx + \Re \int_{\Omega} f \varphi \bar{u}\, dx + \Re \int_{\partial \Omega} g \varphi \bar{u}\, d\sigma.
	\end{multline*}
	Observe that since $\varphi\in W^{2,\infty}_\textup{loc}(\R^n),$ in particular $\Delta \varphi\in L^\infty_\textup{loc}(\R^n),$ moreover, by Sobolev embeddings, $\nabla \varphi$ is a lipschitz function and thus the trace of $\nabla \varphi$ is everywhere defined on $\partial \Omega.$ These facts ensure that the integrals in the previous identity are well-defined and finite.
	
Recalling that~$\Omega$ is the upper half-space~\eqref{upper},
the outer normal~$\eta$ to the boundary satisfies $\eta=(0,0,\dots,0,-1)$.
Taking $\varphi:=1$ and $\varphi:=\abs{x}$, 
we get~\eqref{eq:first_Omega_const} and~\eqref{eq:first_Omega_x}, respectively. In passing, notice that for both the choices of $\varphi$ the boundary term $\int_{\partial \Omega}{ |u|^2 \nabla \varphi\cdot \eta \, d\sigma}$ vanishes, in particular when $\varphi=\abs{x}$ this is a consequence of the  validity of the orthogonality condition $x\cdot \eta=0$ on the boundary $\partial \Omega.$

Equation~\eqref{eq:first_Omega_im_const} and~\eqref{eq:first_Omega_im_x} 
are obtained similarly to the previous case: choose in~\eqref{eq:weak_resolvent_gen} 
$v:= \psi u,$ with $\psi\colon \R^n \to \R,$ $\psi \in W^{1,\infty}_\textup{loc}(\R^n)$ being 
a \emph{radial} function. Now, instead of taking the real part of the resulting identity, one takes the imaginary part obtaining
\begin{equation*}
	\Im \int_{\Omega} \bar{u}\nabla u\cdot \nabla \psi\, dx 
		 +  \int_{\partial \Omega} \Im \alpha \psi \abs{u}^2\, d\sigma
		= \Im \lambda \int_{\Omega} \psi \abs{u}^2\, dx + \Im \int_{\Omega} f \psi \bar{u}\, dx + \Im \int_{\partial \Omega} g \psi \bar{u}\, d\sigma.
\end{equation*}
Finally, we choose $\psi:=1$ and $\psi:=\abs{x}$, respectively.
 
The remaining identity~\eqref{eq:second_Omega} is formally obtained 
by plugging into~\eqref{eq:weak_resolvent_gen} the multiplier
\begin{equation}\label{Morawetz_multiplier}
	v:=[\Delta, \phi]u\\
		=\Delta \phi u + 2\nabla \phi \cdot \nabla u
  \qquad \mbox{with} \qquad
  \phi(x) := |x|^2,
\end{equation}
taking the real part and integrating by parts. However, such~$v$ does not need to belong to $H^1(\R^n)$;
in fact, the unboundedness of~$\phi$ does not pose any problems 
because the support of~$u$ is assumed to be compact,
but~$\nabla u$ does not necessarily belong to $H^1(\R^n)$ 
(unless we strengthen the hypothesis about~$\alpha$,
for instance to $\alpha \in W_\mathrm{loc}^{1,\infty}(\partial\Omega)$).

 For reader's convenience we provide now the \emph{formal} derivation of identity~\eqref{eq:second_Omega}. The rigorous proof will be given immediately below.  
Choosing $v=\Delta \phi u + 2 \nabla \phi \cdot \nabla u=\Delta \phi u + 2 \partial_k \phi \partial_k u$ in~\eqref{eq:weak_resolvent_gen} (notice that here the Einstein summation convention is adopted) and using the Leibniz rule for weak derivatives in the first term of~\eqref{eq:weak_resolvent_gen} we get
\begin{multline*}
	\int_\Omega \Delta \phi \abs{\nabla u}^2\, dx 
	+ \int_\Omega \nabla u \cdot \nabla \Delta \phi \bar{u}\, dx
	+2 \int_\Omega \nabla u \cdot (\nabla \partial_k \phi\, \partial_k \bar{u})\, dx
	+2 \int_\Omega \nabla u \cdot (\partial_k \phi\, \nabla \partial_k \bar{u})\, dx\\
	 + \int_{\partial \Omega} \alpha \Delta \phi \abs{u}^2\, d\sigma
	  +2 \int_{\partial \Omega} \alpha \partial_k \phi\, u\, \partial_k \bar{u}\, dx
	  = \lambda \int_\Omega \Delta \phi \abs{u}^2\, dx 
	  + 2\lambda \int_\Omega \partial_k \phi\, u\, \partial_k \bar{u}\, dx\\
	  + \int_\Omega f(\Delta \phi \bar{u} + 2 \nabla\phi\cdot \nabla \bar{u})\, dx 
	  + \int_{\partial \Omega} g(\Delta \phi \bar{u} + 2 \nabla\phi\cdot \nabla \bar{u})\, d\sigma. 
\end{multline*} 
Taking the real part of the resulting identity and using again that $\Re(\bar{u}\nabla u)=\frac{1}{2} \nabla \abs{u}^2$ one gets
\begin{multline*}
	\int_\Omega \Delta \phi \abs{\nabla u}^2\, dx 
	+\frac{1}{2} \int_\Omega  \nabla \Delta \phi \cdot \nabla \abs{u}^2\, dx
	+2 \int_\Omega \nabla u \cdot (\nabla \partial_k \phi\, \partial_k \bar{u})\, dx
	+ \int_\Omega \partial_k \abs{\nabla u}^2 \partial_k \phi\, dx\\
	 + \int_{\partial \Omega} \Re \alpha \Delta \phi \abs{u}^2\, d\sigma
	  +2 \Re \int_{\partial \Omega} \alpha \partial_k \phi\, u\, \partial_k \bar{u}\, dx\\
	  = \Re \lambda \int_\Omega \Delta \phi \abs{u}^2\, dx 
	  + \Re \lambda \int_\Omega \partial_k \phi\, \partial_k \abs{u}^2\, dx
	  -	2 \Im \lambda  \Im \int_\Omega \partial_k \phi\, u\, \partial_k \bar{u}\, dx\\
	  + \Re \int_\Omega f(\Delta \phi \bar{u} + 2 \nabla\phi\cdot \nabla \bar{u})\, dx 
	  +\Re \int_{\partial \Omega} g(\Delta \phi \bar{u} + 2 \nabla\phi\cdot \nabla \bar{u})\, d\sigma,
\end{multline*}
where in the part involving the eigenvalue $\lambda$ we have used that given $z_1, z_2 \in \C,$ then $\Re(z_1 z_2)= \Re z_1 \Re z_2 - \Im z_1 \Im z_2.$
 Integrating by parts the previous identity leads
\begin{multline*}
	- \frac{1}{2}\int_\Omega \Delta^2 \phi \abs{u}^2\, dx
	+ \frac{1}{2} \int_{\partial \Omega} \abs{u}^2 \nabla \Delta \phi \cdot \eta \, d\sigma
	+ 2\int_\Omega \nabla u\cdot \nabla^2 \phi \cdot \nabla \bar{u}\, dx
	+ \int_{\partial \Omega} \abs{\nabla u}^2 \nabla \phi \cdot \eta\, d\sigma\\
	+ \int_{\partial \Omega} \Re \alpha \Delta \phi \abs{u}^2\, d\sigma
	  +2 \Re \int_{\partial \Omega} \alpha \nabla \phi \cdot (u \nabla \bar{u})\, dx
	  =\Re \lambda \int_{\partial \Omega} \abs{u}^2 \nabla \phi\cdot \eta \, d\sigma
	  -	2 \Im \lambda  \Im \int_\Omega \nabla \phi \cdot (u \nabla \bar{u})\, dx\\
	  + \Re \int_\Omega f(\Delta \phi \bar{u} + 2 \nabla\phi\cdot \nabla \bar{u})\, dx 
	  +\Re \int_{\partial \Omega} g(\Delta \phi \bar{u} + 2 \nabla\phi\cdot \nabla \bar{u})\, d\sigma.
\end{multline*} 
Choosing now $\phi(x):=\abs{x}^2$ and using that
\begin{equation*}
	\nabla \phi= 2x,\qquad
	\nabla^2 \phi= 2\, \textup{Id},\qquad
	\Delta \phi=2n,\qquad 
	\nabla \Delta \phi=0,\qquad
	\Delta^2 \phi=0, 
\end{equation*}
one immediately gets~\eqref{eq:second_Omega} (after multiplication by $1/2$). Notice that also in this situation the boundary terms $\int_{\partial \Omega} \abs{\nabla u}^2 \nabla \phi \cdot \eta \, d\sigma$ and $\Re\lambda \int_{\partial \Omega} \abs{u}^2 \nabla \phi \cdot \eta\, d\sigma$ vanish as a consequence of the orthogonality condition $x\cdot \eta=0$ on the boundary $\partial \Omega.$ 
In view of the above, we stress that the choice of the multipliers in the previous identities represents a delicate step and it is strongly based on the geometry of the domain we are working in. One characterising feature is that our choice of $\varphi, \psi$ and $\phi$ introduces an orthogonality condition which allows us to discard possibly troublesome terms defined on the boundary of our domain. Nonetheless, one observes that such a condition easily generalises to more general sectors than the mere half-space.

In order to \emph{rigorously} prove~\eqref{eq:second_Omega}, our idea is to replace~\eqref{Morawetz_multiplier} by its regularised version
\begin{equation}\label{Morawetz_multiplier_reg}
v:=\Delta \phi u + \nabla \phi \cdot[\nabla^\delta u + \nabla ^{-\delta} u]
=\Delta \phi u + \partial_k \phi\, [\partial_k^\delta u + \partial_k^{-\delta}u],
\end{equation}
where $\delta \in \R\setminus\{0\}$,
$\nabla^\delta \psi :=(\partial_1^\delta \psi, \dots, \partial_n^\delta \psi)$ 
and 
\begin{equation*}
  \partial_k^\delta u(x)
  :=\frac{\tau_k^\delta u(x)-u(x)}{\delta}
  \qquad \mbox{with} \qquad
  \tau_k^\delta u(x) := u(x+ \delta e_k),
\end{equation*}
is the standard \emph{difference quotient} of~$u$
(it makes sense if we recall the convention that by $u \in H^1(\R^n)$
we understand the extension~$Eu$ of $u \in H^1(\Omega)$,
where~$E$ is the extension operator).	 
For a moment, we proceed in a greater generality by considering~$\phi$
to be an arbitrary smooth function $\phi:\R^n \to \R$.
%
%

We refer to \cite[Sec.~5.8.2]{Evans} or~\cite[Sec.~10.5]{Leoni}
for basic facts about the difference quotients.
Here we only point out the following important property,
which we did not find in these references
(however, \emph{cf}.~\cite[Thm.~10.55]{Leoni}).
For the proof, one can use the fundamental theorem of calculus
and the theorem of Lusin. 
\begin{proposition}\label{Prop.Lusin}
Let $\psi \in W^{1,p}(\R^n)$ with $1 \leq p < \infty$.
Then the strong convergence
\begin{equation*}
  \partial_k^{\delta}\psi
  \xrightarrow[\delta \to 0]{}
  \partial_k\psi
  \qquad \mbox{in} \qquad
  L^p(U)
\end{equation*}
holds true for every subdomain $U \subset\subset \R^n$.
\end{proposition}

We plug~\eqref{Morawetz_multiplier_reg} into~\eqref{eq:weak_resolvent_gen} 
and take the real part. Below, for the sake of clarity, 
we consider each integral of the resulting identity separately.

\subsubsection*{$\bullet$ Kinetic term}
Let us start with the ``kinetic'' part of~\eqref{eq:weak_resolvent_gen}: 
\begin{equation*}
	K:=\Re \int_\Omega \nabla u \cdot \nabla \bar{v}\, dx.
\end{equation*}
Using
\begin{equation*}
	\partial_l \bar{v}= (\partial_l \Delta \phi) \bar{u} + \Delta \phi \partial_l \bar{u} + \partial_{lk} \phi\, [\partial_k^\delta \bar{u} + \partial_k^{-\delta}\bar{u}] + \partial_k \phi\, [\partial_l\partial_k^{\delta}\bar{u} + \partial_l \partial_k^{-\delta} \bar{u}],
\end{equation*}
we write $K=K_1+K_2+K_3+K_4$ with 
\begin{align*}
K_1&:=\Re \int_\Omega \partial_l u (\partial_l \Delta \phi) \bar{u}\, dx, 
&
K_2&:=\int_{\Omega} \abs{\nabla u}^2 \Delta \phi\, dx,
\\
K_3&:=\Re \int_{\Omega} \partial_{lk}\phi\, \partial_l u\, [\partial_k^\delta \bar{u} + \partial_k^{-\delta} \bar{u}]\, dx, 
& 
K_4&:=\Re \int_{\Omega} \partial_k \phi \, \partial_l u\, [\partial_k^\delta \partial_l \bar{u} + \partial_{k}^{-\delta}\partial_l \bar{u}]\, dx.
\end{align*}
Integrating by parts in $K_1$ gives
\begin{equation*}
	K_1=-\frac{1}{2}\int_{\Omega} \Delta^2 \phi \abs{u}^2\, dx + \frac{1}{2} \int_{\partial \Omega} \abs{u}^2 \nabla \Delta \phi \cdot \eta\, d\sigma.
\end{equation*}
Now we consider $K_4.$ 
Using the formula
\begin{equation}\label{real}
  2\Re(\bar{\psi} \partial_k^\delta \psi)
  =\partial_k^\delta \abs{\psi}^2 -\delta \abs{\partial_k^\delta \psi}^2
\end{equation}
valid for every $\psi:\R^n \to \C$,
we write $K_4=K_{4,1} + K_{4,2}$ with
(summation both over~$k$ and~$l$)
\begin{equation*}
	K_{4,1}:= \frac{1}{2} \int_{\Omega} \partial_k \phi \{\partial_k^\delta \abs{\partial_l u}^2 + \partial_k^{-\delta} \abs{\partial_l u}^2\}\, dx,\quad \text{and} \quad
	K_{4,2}:=-\frac{\delta}{2} \int_{\Omega} \partial_k \phi \{ \abs{\partial_k^\delta \partial_l u}^2 - \abs{\partial_k^{-\delta} \partial_l u}^2 \}\, dx.
\end{equation*}
It is well known that the \emph{integration-by-parts 
formula for difference quotients} (see \cite[Sec.~5.8.2]{Evans})
\begin{equation*}
  \int_{\R^n} \varphi \ \partial_k^{\delta}\psi \, dx
  = - \int_{\R^n} (\partial_k^{-\delta} \varphi) \ \psi \, dx
\end{equation*}
holds true for every $\varphi,\psi \in L^2(\R^n)$.
Here we use the following variant in the half-space
\begin{equation}\label{int_by_parts}
\int_{\Omega} \varphi \partial_k^\delta \psi\, dx
=- \int_{\Omega} (\partial_k^{-\delta} \varphi) \psi\, dx 
-\delta_{k,n} \frac{1}{\delta} \int_0^{\delta} \int_{\R^{n-1}} 
(\tau_k^{-\delta} \varphi) \psi\, dx,
\end{equation}
where $\varphi,\psi \in L^2(\Omega)$
and $\delta_{k,n}$ denotes the Kronecker symbol.
Consequently,	 
\begin{multline*}
	K_{4,1}=-\frac{1}{2}\int_{\Omega} \{\partial_k^{-\delta}\partial_k \phi + \partial_k^{\delta}\partial_k \phi \} \abs{\nabla u}^2\, dx\\
	-\delta_{k,n}\frac{1}{2\delta} \int_0^\delta \int_{\R^{n-1}} (\tau_k^{-\delta}\partial_k \phi)\abs{\partial_l u}^2\, dx
	-\delta_{k,n}\frac{1}{2\delta} \int_{-\delta}^0 \int_{\R^{n-1}} (\tau_k^{\delta}\partial_k \phi)\abs{\partial_l u}^2\, dx.  
\end{multline*}
At the same time,
making explicit the difference quotient and changing variable in $K_{4,2}$ gives
(summation both over~$k$ and~$l$)
\begin{equation*}
	K_{4,2}=-\frac{\delta}{2} \int_{\Omega} \{\partial_k \phi - (\tau_{k}^\delta \partial_k \phi)\} \abs{\partial_k^\delta \partial_l u}^2\, dx + \delta_{k,n} \frac{\delta}{2}\int_{-\delta}^0 \int_{\R^{n-1}} (\tau_k^\delta \partial_k \phi)\abs{\partial_k^\delta \partial_l u}^2\, d\sigma. 
\end{equation*}
Now we choose the multiplier $\phi(x):=\abs{x}^2$ and observe that
\begin{equation}\label{eq:phi_derivatives}
	\partial_k \phi=2x_k, \qquad \partial_{lk} \phi=2\delta_{k,l}, \qquad \partial_k^{\pm\delta}\partial_k\phi=2,\qquad \nabla \Delta \phi=0, \qquad \Delta^2\phi=0.
\end{equation}
Consequently,
\begin{equation*}
	K_1=0,\qquad
	K_2=2n \int_{\Omega} \abs{\nabla u}^2\, dx,\qquad
	K_3=2 \Re \int_{\Omega} \partial_l u\,[\partial_l^\delta \bar{u} + \partial_l^{-\delta}\bar{u}]\, dx,
\end{equation*}
and
\begin{multline*}
K_{4}=-2n\int_{\Omega} \abs{\nabla u}^2\, dx-\frac{1}{\delta} \int_{-\delta}^\delta \int_{\R^{n-1}} x_n \abs{\nabla u}^2\, dx +\int_0^\delta \int_{\R^{n-1}} \abs{\nabla u}^2\, dx - \int_{-\delta}^0 \int_{\R^{n-1}} \abs{\nabla u}^2\, dx
\\ +\int_{\Omega} \abs{\tau_k^\delta \nabla u -\nabla u}^2\, dx + \frac{1}{\delta}\int_{-\delta}^0 \int_{\R^{n-1}} (x_n + \delta)\abs{\tau_n^\delta \nabla u - \nabla u}^2\, dx.
\end{multline*}
Using the absolutely continuity of Lebesgue integral and the $L^2$-continuity of translation, one arrives at
\begin{equation}\label{kinetic_approx}
	K= 2\Re \int_{\Omega} \partial_l u\, [\partial_l^\delta \bar{u} + \partial_l^{-\delta} \bar{u}] + \varepsilon(\delta),
\end{equation}
where $\varepsilon(\delta) \to 0$ as $\delta \to 0$.

\subsubsection*{$\bullet$ Source term}
Let us now consider simultaneously
the ``source'' and ``eigenvalue''
parts of~\eqref{eq:weak_resolvent_gen}, that is,
\begin{equation*}
	F:=\Re \left( \lambda \int_{\Omega} u \bar{v}\, dx  + \int_\Omega f \bar{v}\, dx \right).
\end{equation*}
This can be written as $F=F_1+F_2+F_3+F_4$ with
\begin{align*}
	F_1&:=\Re\lambda \int_{\Omega} \Delta \phi \abs{u}^2\, dx,
&
	F_2&:=\Re \lambda \Re \int_\Omega \partial_k \phi \, u [\partial_k^\delta \bar{u} + \partial_k^{-\delta} \bar{u}]\, dx,
\\
	F_3&:=-\Im \lambda \Im \int_\Omega \partial_k \phi \, u [\partial_k^\delta \bar{u} + \partial_k^{-\delta} \bar{u}]\, dx ,
&
	F_4&:= \Re \int_\Omega f\{\Delta \phi \bar{u} + \partial_k \phi \, [\partial_k^\delta \bar{u} + \partial_k^{-\delta} \bar{u}]\}\, dx.
\end{align*}
Applying~\eqref{real}, we further split $F_2=F_{2,1} +F_{2,2}$, where 
\begin{equation*}
	F_{2,1}:=\frac{1}{2} \Re \lambda \int_{\Omega} \partial_k \phi\, \{\partial_k^\delta \abs{u}^2 + \partial_k^{-\delta} \abs{u}^2 \}\, dx\quad \text{and}\quad
	F_{2,2}:=-\frac{\delta}{2} \Re \lambda \int_{\Omega} \partial_k \phi\, \{ \abs{\partial_k^\delta u}^2 - \abs{\partial_k^{-\delta} u}^2 \}\, dx.
\end{equation*}
Using the integrating-by-parts formula~\eqref{int_by_parts}, we get
\begin{multline*}
	F_{2,1}=-\frac{1}{2}\Re\lambda \int_{\Omega} \{\partial_k^{-\delta} \partial_k \phi + \partial_k^{\delta} \partial_k \phi\}\abs{u}^2\, dx\\
	-\delta_{k,n} \frac{1}{2\delta} \Re \lambda \int_0^\delta \int_{\R^{n-1}} (\tau_k^{-\delta}\partial_k\phi)\abs{u}^2\, dx -\delta_{k,n} \frac{1}{2\delta} \Re \lambda \int_{-\delta}^0 \int_{\R^{n-1}} (\tau_k^\delta \partial_k \phi) \abs{u}^2\, dx.
\end{multline*}
Choosing $\phi(x):=\abs{x}^2$ 
in the previous identities and using~\eqref{eq:phi_derivatives} gives
\begin{align*}
	& F_1=2n \Re\lambda \int_{\Omega} \abs{u}^2\, dx\\
	& 
	\begin{aligned}
		F_2=&-2n \Re \lambda \int_{\Omega} \abs{u}^2\, dx -\frac{1}{\delta} \int_{-\delta}^\delta \int_{\R^{n-1}} x_n \abs{u}^2\, dx +\int_0^\delta \int_{\R^{n-1}} \abs{u}^2\, dx -\int_{-\delta}^0 \int_{\R^{n-1}} \abs{u}^2\, dx\\
  			& 2\delta \Re \lambda \int_\Omega x_k \{\abs{\partial_k^\delta u}^2 - \abs{\partial_k^{-\delta} u}^2\}\, dx, 
	\end{aligned}\\
	&F_3=-2\Im \lambda \Im \int_\Omega x_k u\, [\partial_k^\delta \bar{u} + \partial_k^{-\delta} \bar{u}]\, dx,\\
	&F_4= \Re \int_{\Omega} f \{2n \bar{u} +2 x_k [\partial_k^\delta \bar{u} + \partial_k^{-\delta} \bar{u}]\}\, dx.
\end{align*}
Using again the absolutely continuity of the Lebesgue integral 
and the strong $L^2$-convergence of the difference quotients
(see Proposition~\ref{Prop.Lusin}), one has
\begin{equation}\label{eigen-source_approx}
	F=-2\Im \lambda \Im \int_\Omega x_k u\, [\partial_k^\delta \bar{u} + \partial_k^{-\delta} \bar{u}]\,dx +  \Re \int_{\Omega} f \{2n \bar{u} +2 x_k [\partial_k^\delta \bar{u} + \partial_k^{-\delta} \bar{u}]\}\,dx+ \varepsilon(\delta),
\end{equation}
where $\varepsilon(\delta) \to 0$ as $\delta \to 0$.

\subsubsection*{$\bullet$ Boundary-potential term}
Let us now consider the contribution 
of the ``potential'' part of~\eqref{eq:weak_resolvent_gen}, that is, 
\begin{equation*}
	J:=\Re \int_{\partial \Omega} \alpha u \bar{v}\, d\sigma.
\end{equation*}
This can be written as $J=J_1+J_2$ with
\begin{equation*}
	J_1:= \int_{\partial \Omega} \Re \alpha \Delta \phi \abs{u}^2\, d\sigma\qquad \text{and} \qquad
	J_2:= \Re \int_{\partial \Omega} \alpha\, \partial_k\phi\, u\,[\partial_k^\delta \bar{u} + \partial_k^{-\delta}\bar{u}]\, d\sigma.
\end{equation*}
Choosing $\phi(x):=\abs{x}^2$ in the previous identities and using~\eqref{eq:phi_derivatives} we get
\begin{equation*}
	J_1=2n\int_{\partial \Omega} \Re\alpha \abs{u}^2\, d\sigma\qquad \text{and} \qquad
	J_2=2 \Re \int_{\partial \Omega} \alpha \, x_k\,  u\, [\partial_k^\delta \bar{u} + \partial_k^{-\delta}\bar{u}]\, d\sigma.
\end{equation*}
Hence
\begin{equation}\label{potential_approx}
	J=2n\int_{\partial \Omega} \Re\alpha \abs{u}^2\, d\sigma + 2\Re \int_{\partial \Omega} \alpha \, x_k\,  u\, [\partial_k^\delta \bar{u} + \partial_k^{-\delta}\bar{u}]\, d\sigma.
\end{equation}

\subsubsection*{$\bullet$ Boundary-source term}
Let us conclude by considering the ``boundary-source'' part
of~\eqref{eq:weak_resolvent_gen}, that is, 
\begin{equation*}
	G:=\Re \int_{\partial \Omega} g\bar{v}\, d\sigma.
\end{equation*}
With the choice $\phi(x):=\abs{x}^2$ 
and using~\eqref{eq:phi_derivatives}, 
we have
\begin{equation}\label{boundary_source_approx}
	G= \Re \int_{\partial \Omega} g\{2n \bar{u} + 2 x_k \, [\partial_k^\delta \bar{u} + \partial_k^{-\delta} \bar{u}]\}\, d\sigma.
\end{equation}

\subsubsection*{$\bullet$ Passing to the limit $\delta \to 0$}
Now since $u\in H^{3/2}(\R^n)$ is such that $\Delta u \in L^2(\R^n),$ using the continuity of the trace operator which holds true under these assumptions (see Lemma~\ref{lemma:Neumann_3/2epsilon}) 
and the strong $L^2$-convergence of difference quotients 
(see Proposition~\ref{Prop.Lusin}),
one has
\begin{align*}
	&\text{~\eqref{kinetic_approx}} \xrightarrow[\delta \to 0]{} 4 \int_\Omega \abs{\nabla u}^2\, dx,\\
	&\text{~\eqref{eigen-source_approx}} \xrightarrow[\delta \to 0]{}-4\Im \lambda \Im \int_\Omega x\cdot (u \nabla \bar{u})\, dx + 2n \Re \int_\Omega f \bar{u}\, dx + 4 \Re \int_\Omega f x\cdot \nabla \bar{u}\, dx,\\
	&~\eqref{potential_approx} \xrightarrow[\delta \to 0]{} 2n \int_{\partial \Omega} \Re\alpha \abs{u}^2\, d\sigma + 4 \Re \int_{\partial \Omega} \alpha\, x\cdot (u \nabla \bar{u})\, d\sigma,\\
	&~\eqref{boundary_source_approx} \xrightarrow[\delta \to 0]{} 2n \Re \int_{\partial \Omega} g \bar{u}\, d\sigma + 4 \Re \int_{\partial \Omega} g\, x \cdot \nabla \bar{u}\, d\sigma. 
\end{align*}
Therefore, passing to the limit $\delta \to 0$
in~\eqref{eq:weak_resolvent_gen} 
and multiplying the resulting identity by $1/2$,
one obtains~\eqref{eq:second_Omega}.

\medskip
This concludes the proof of Lemma~\ref{lemma:5identities}.
\qed

\subsection*{Acknowledgment}
%
We are grateful to Luca Fanelli, Rupert Frank and Luis Vega for useful discussions.
The research of D.K.\ was partially supported 
by the GACR grant No.\ 18-08835S.

%
%
	
\providecommand{\bysame}{\leavevmode\hbox to3em{\hrulefill}\thinspace}
\providecommand{\MR}{\relax\ifhmode\unskip\space\fi MR }
\providecommand{\MRhref}[2]{%
  \href{http://www.ams.org/mathscinet-getitem?mr=#1}{#2}
}
\providecommand{\href}[2]{#2}

\end{document}